\documentclass[11pt]{amsart}
\usepackage{amsmath}
\usepackage{bbding}
\usepackage{tikz}
\usepackage{amsmath,amssymb}
\theoremstyle{plain}
\newtheorem{theorem}{Theorem}[section]
\newtheorem{lemma}[theorem]{Lemma}

\theoremstyle{definition}

\newtheorem{example}[theorem]{Example}
\numberwithin{equation}{section}

\def\be{\begin{equation}}
\def\ee{\end{equation}}

\begin{document}

\title[Uniformly Degenerate Elliptic Equations]
{Optimal Boundary Regularity\\ for Uniformly Degenerate Elliptic Equations}

\author[Han]{Qing Han}
\address{Department of Mathematics\\
University of Notre Dame\\
Notre Dame, IN 46556, USA} \email{qhan@nd.edu}
\author[Xie]{Jiongduo Xie}
\address{Beijing International Center for Mathematical Research\\
Peking University\\
Beijing, 100871, China}
\email{2001110018@stu.pku.edu.cn}


\maketitle


\section{Introduction}\label{sec-OptimalRegularity-linear-intro} 

In this survey paper, we study the optimal regularity of solutions to 
uniformly degenerate elliptic equations in bounded domains and 
establish the H\"older continuity of solutions and their derivatives 
up to the boundary. 

To introduce uniformly degenerate elliptic equations, 
we briefly review the Loewner-Nirenberg problem in the Euclidean space setting.  
Let $\Omega\subseteq \mathbb{R}^{n}$ be a bounded domain, for  $n\ge 3$. 
In the Loewner-Nirenberg problem, we will find a positive solution of
\begin{align}
\label{eq-ch1-LN-MainEq}\begin{split}
\Delta u  &= \frac14n(n-2) u^{\frac{n+2}{n-2}} \quad\text{in }\Omega,\\
u&=\infty\quad\text{on }\partial \Omega.
\end{split}\end{align}
Geometrically, $u^{\frac{4}{n-2}}g_E$
is a complete metric with a constant scalar curvature $-n(n-1)$ in $\Omega$, 
where $g_E$ is the flat metric in the Euclidean space $\mathbb R^n$. 

In a pioneering work,
Loewner and Nirenberg \cite{Loewner&Nirenberg1974}
studied the existence of solutions of \eqref{eq-ch1-LN-MainEq}  
and 
characterized the blowup rate of solutions near the boundary. Specifically, 
let $\Omega$ be a bounded domain with a $C^{2}$-boundary and 
$\rho$ be a $C^{2}(\bar\Omega)$-defining function of $\Omega$; 
namely, $\rho$ satisfies $\rho>0$ in $\Omega$ and $\rho=0$ and
$|\nabla\rho|=1$ on $\partial\Omega$. 
Then,  
\eqref{eq-ch1-LN-MainEq} 
admits a unique positive solution $u\in C^\infty(\Omega)$ and
\begin{equation}\label{eq-ch1-basic-estimate-LN}
| \rho^{\frac{n-2}{2}}u-1|\leq C\rho\quad\text{in }\Omega,\end{equation} 
where $C$
is a positive constant depending only on $n$ and $\Omega$.
In particular, $u$ blows up near $\partial\Omega$ at the rate $\rho^{-(n-2)/2}$. 

In view of the estimate \eqref{eq-ch1-basic-estimate-LN}, we introduce the function $v$ such that 
\begin{equation}\label{eq-ch1-LN-Relation} 
u=\rho^{-\frac{n-2}2}(1+v).\end{equation}
A straightforward computation implies that $v$ satisfies 
\begin{align}\label{eq-ch1-LN-NonlinearOperator}\begin{split}
&\rho^2\Delta v-(n-2)\rho\nabla \rho\cdot \nabla v
-\Big[n+\frac14n(n-2)(1-|\nabla\rho|^2)+\frac12(n-2)\rho\Delta \rho \Big]v\\
&\qquad\qquad=F(v)
\quad\text{in }\Omega,
\end{split}\end{align}
where
\begin{align}\label{eq-ch1-LN-NonlinearOperator-1}\begin{split}
F(v)&=\frac12(n-2)\rho\Delta \rho+\frac14n(n-2)(1-|\nabla \rho|^2)
\\&\qquad
+\frac14n(n-2)\Big[(1+v)^{\frac{n+2}{n-2}}-1-\frac{n+2}{n-2}v\Big].
\end{split}\end{align}
By the estimate \eqref{eq-ch1-basic-estimate-LN}, we have 
\begin{align*}
|v|\le C\rho\quad\text{in }\Omega.\end{align*}
In particular, $v$ is continuous up to the boundary and $v=0$ on $\partial\Omega$.

The equation \eqref{eq-ch1-LN-NonlinearOperator} is a semilinear elliptic equation. 
In the expression of $F$ given by \eqref{eq-ch1-LN-NonlinearOperator-1}, 
the first two terms are considered to be 
nonhomogeneous terms, and the third term is a nonlinear expression of $v$, 
{\it without} linear terms. 
The left-hand side of \eqref{eq-ch1-LN-NonlinearOperator} 
consists of all linear terms of $v$, with 
a factor $\rho^2$ in all terms of the second derivatives of $v$ 
and a factor $\rho$ in all terms of the first derivatives of $v$. 
The factor $\rho^2$ for the second derivatives of $v$ 
causes degeneracy of the linear operator in the left-hand side
of \eqref{eq-ch1-LN-NonlinearOperator}. 
The degeneracy of this type is referred to as the uniform degeneracy. 

Uniformly degenerate elliptic equations appear or can be uncovered as above
in many problems. 
In addition to the Loewner-Nirenberg problem and its nonlinear versions \cite{Loewner&Nirenberg1974,AvilesMcOwen1988,Mazzeo1991,ACF1992CMP, Kichenassamy2004JFA, Kichenassamy2005JFA, GonzalezLi2018, LiNguyenXiong2023}, 
these problems include 
complex Monge-Amp\`ere equations
\cite{Fefferman1976, ChengYau1980CPAM, LeeMelrose1982, HanJiang2024}, 
affine hyperbolic spheres 
\cite{ChengYau1977,JianWang2013JDG,JianWangZhao2017JDE,JianLuWang2022China}, 
the Bergman Laplacian \cite{Graham1983-1, Graham1983-2}, 
minimal graphs in the hyperbolic space  
\cite{Lin1989Invent,Tonegawa1996MathZ,Graham&Witten1999,Lin2012Invent,HanShenWang16CalVar,HanJiang2023}, 
singular stochastic control problems 
\cite{LasryLions1989, LeonoriPorretta2011, JianLuWang2022China}, 
conformally compact Einstein metrics
\cite{GrahamLee1991, Lee1995, OB2000, Graham2000, OB2002, Chrusciel2005, Lee2006, Hellimell2008, OB2016}, 
and proper harmonic maps between hyperbolic spaces and other spaces
\cite{LiTam1991,LiTam1993,LiTam1993Ind,Donnelly1994,LiNi2000,LiSimon2007,Yin2007,ChenLi2023}. 
Many of these topics have been extensively studied. 
A few references here are closely related to the existence and regularity. 
After introducing the degenerate elliptic equation for the graphs of the minimal surfaces in the hyperbolic space, 
Lin \cite{Lin1989Invent} wrote, \lq\lq It may suggest a general study of linear or non-linear degenerate elliptic equations. 
For example, one would like to know for which class of such equations the solutions are as smooth as the boundary data. 
In the case that the loss in derivatives does occur, one would like to know exactly how much is lost."

The study of the regularity of the solutions to the uniformly degenerate elliptic equations
consists of two parts. Take \eqref{eq-ch1-LN-NonlinearOperator} for an example. 
In the first part, we study the optimal regularity up to a certain level 
under an appropriate regularity assumption of the boundary $\partial\Omega$. 
Here, the optimality refers to the fact that the regularity of solutions cannot be improved 
beyond a level determined by the equation 
even if the boundary $\partial\Omega$ is smooth. 
This is a special feature of the uniformly degenerate elliptic equations. 
Graham and Lee \cite{GrahamLee1991} studied the existence and regularity in weighted H\"older spaces.
In the second part, we study behaviors of solutions near the boundary 
beyond the optimal regularity level in the form of the polyhomogeneous expansions. 
Such expansions were originally established by the method of microlocal analysis by Melrose \cite{Melrose1981}, Lee and Melrose \cite{LeeMelrose1982}, 
Mazzeo and Melrose \cite{MazzeoMelrose1987}, and Mazzeo \cite{Mazzeo1991}.

In the study of the optimal regularity, 
the structure of the linear operators plays a significant role. 
In this survey paper, we will present an outline of such study 
for uniformly degenerate linear elliptic equations by 
elliptic PDE methods. 
Our arguments are modified and simplified from earlier approaches 
for the aforementioned nonlinear problems. 
We will formulate our regularity and existence results for functions in the standard 
H\"older spaces, with no additional growth or decay assumptions on nonhomogeneous terms 
or solutions. 
These results form the Schauder theory for uniformly degenerate elliptic equations.

{\it Acknowledgements:} The first author acknowledges the support of NSF Grant DMS-2305038.
The second author acknowledges the support of 
National Key R\&D Program of China Grant 2020YFA0712800.

\section{Main Results}\label{sec-Main-Results}

Let $\Omega$ be a bounded domain in $\mathbb R^n$. 
For a uniformly degenerate linear elliptic operator $L$ to be introduced, we consider the Dirichlet problem 
for the equation 
\begin{equation}\label{eq-ch2-basic-equation}
Lu=f\quad\text{in }\Omega.\end{equation}
Under the standard H\"older regularity assumption on $f$, 
we will study the regularity of solutions 
of the Dirichlet problem. We will not impose any additional assumptions on $f$, such as 
growth or decay near the boundary. 

We first introduce a notion of defining functions. 
Let $\Omega$ be a bounded domain in $\mathbb R^n$ with a $C^{1}$-boundary $\partial\Omega$ 
and $\rho$ be a $C^{1}(\bar\Omega)$-function. 
Then, $\rho$ is a {\it defining function} of $\Omega$
if $\rho>0$ in $\Omega$ and $\rho=0$ and $\nabla \rho\neq 0$ on $\partial\Omega$. We always require 
$$|\nabla\rho|=1\quad\text{on }\partial\Omega.$$ Then, $\nabla \rho|_{\partial\Omega}$ is the inner unit normal vector along $\partial\Omega$. 
In the following, we will fix a defining function $\rho$, with possibly a higher regularity.

Consider the operator 
\begin{align}\label{eq-ch2-def-operator}
L=\rho^2a_{ij}\partial_{ij}+\rho b_i\partial_i+c\quad\text{in }\Omega,\end{align}
where $a_{ij}$, $b_i$, and $c$ are continuous functions in $\bar\Omega$ satisfying $a_{ij}=a_{ji}$ and, 
for any $x\in \Omega$
and $\xi\in\mathbb R^n$, 
\begin{align}\label{eq-ch2-ellipticity}
\lambda|\xi|^2\le a_{ij}(x)\xi_i\xi_j\le\Lambda|\xi|^2,\end{align}
for some positive constants $\lambda$ and $\Lambda$. 
We note that the operator $L$ is not uniformly elliptic in $\bar \Omega$ and, in fact, is 
degenerate along the boundary $\partial\Omega$, due to the presence of the factor $\rho^2$ in the 
second-order terms. The operator $L$ is called to be {\it uniformly degenerate elliptic}. 

We first examine an example. 

\begin{example}\label{exa-ch2-solutions-not-regular} 
Let $\Omega$ be a bounded domain in $\mathbb R^n$, 
with a $C^\infty$-boundary $\partial\Omega$ and a $C^\infty(\bar\Omega)$-defining function $\rho$. 
Consider constants $a, b, c$, and $s$, with $a>0$ and $s>0$, such that 
\begin{equation}\label{eq-ch2-root-condition}
as(s-1)+bs +c=0.\end{equation}
Set 
\begin{align}\label{eq-ch2-def-operator-example}L=a\rho^2\Delta+b\rho\nabla\rho\cdot\nabla+c.\end{align}
In other words, we take $a_{ij}=a\delta_{ij}$ and $b_i=b\partial_i\rho$ in \eqref{eq-ch2-def-operator}. 
We now consider three cases. 

\smallskip 
{\it Case 1: Singular monomial factors.} We assume that $s$ is not 
an integer.  
Take a function $\psi\in C^\infty(\bar\Omega)$. 
A straightforward calculation implies
\begin{align*} L(\psi\rho^s)=f,\end{align*}
where 
\begin{align}\label{eq-ch2-example-expression-f}\begin{split}
f&=\rho^s\big\{\big[as(s-1)+bs+c\big]\psi+\big[as(s-1)+bs\big](|\nabla \rho|^2-1)\psi\\
&\qquad+\big[as\psi\Delta\rho+(2as+b)\nabla\psi\cdot\nabla\rho\big]\rho+a\Delta\psi\rho^{2}\big\}.\end{split}\end{align}
All functions in the right-hand side are smooth except $\rho^s$. 
Note that the first term in the right-hand side of \eqref{eq-ch2-example-expression-f} is zero 
by \eqref{eq-ch2-root-condition}.  
By $|\nabla \rho|^2=1$ on $\partial\Omega$,  
we can write $f=\eta\rho^{s+1}$, 
for some function $\eta\in C^\infty(\bar\Omega)$. 
In summary, $u=\psi\rho^s$ is a solution of \eqref{eq-ch2-basic-equation}, for some $f=\eta\rho^{s+1}$. Write 
$s=k+\alpha$, for some integer $k\ge 0$ and constant $\alpha\in (0,1)$. 
For such $u$ and $f$, we have $f\in C^{k+1, \alpha}(\bar\Omega)$ and  
$u\in C^{k,\alpha}(\bar\Omega)$, but $u\notin C^{k,\beta}(\bar\Omega)$, 
for any $\beta\in (\alpha,1)$. 

In fact, by choosing $\psi$ appropriately, we can improve the regularity 
of $f$ in \eqref{eq-ch2-example-expression-f}. For a function 
$\psi_0\in C^\infty(\bar\Omega)$ and a positive integer $m$, consider 
$$\psi=\sum_{l=0}^m\psi_l\rho^l.$$
We now assume in addition that $a\mu(\mu-1)+b\mu +c\neq 0$, for $\mu=s+1, \cdots, s+m$. 
Substituting $\psi$ in \eqref{eq-ch2-example-expression-f} and choosing $\psi_1, \cdots, \psi_m$ successively, 
we can write $f=\eta \rho^{s+m+1}$ for some  $\eta\in C^\infty(\bar\Omega)$. Hence, 
$f\in C^{k+m+1, \alpha}(\bar\Omega)$ and  
$u\in C^{k,\alpha}(\bar\Omega)$, but $u\notin C^{k,\beta}(\bar\Omega)$, 
for any $\beta\in (\alpha,1)$. 

\smallskip 
{\it Case 2: Singular logarithmic factors.} We assume that $s$ is an integer.  
Take a function $\psi\in C^\infty(\bar\Omega)$. 
A straightforward calculation implies
\begin{align*} L(\psi\rho^s\log\rho)=f,\end{align*}
where
\begin{align}\label{eq-ch2-example-expression-f-integer}\begin{split}
f&=\rho^s\log \rho\big\{\big[as(s-1)+bs+c\big]\psi+\big[as(s-1)+bs\big](|\nabla \rho|^2-1)\psi\\
&\qquad+\big[as\psi\Delta\rho+(2as+b)\nabla\psi\cdot\nabla\rho\big]\rho
+a\Delta\psi\rho^{2}\big\}\\
&\qquad +\rho^{s}\big\{\big[a(2s-1)+b\big]\psi|\nabla\rho|^2
+\big[a\psi\Delta\rho+2a\nabla\psi\cdot\nabla\rho\big]\rho\big\}.\end{split}
\end{align}
All functions in the right-hand side are smooth except $\log\rho$. 
Note that the first term in the right-hand side 
is zero 
by \eqref{eq-ch2-root-condition}.  
By $|\nabla \rho|^2=1$ on $\partial\Omega$, 
we can write $f=\eta_1\rho^{s+1}\log\rho+\eta_2$, 
for some functions $\eta_1, \eta_2\in C^\infty(\bar\Omega)$. 
In summary, $u=\psi\rho^s\log \rho$ is a solution of \eqref{eq-ch2-basic-equation}, 
for the above $f$. For such $u$ and $f$, we have 
$f\in C^{s, \alpha}(\bar\Omega)$ and  
$u\in C^{s-1,\alpha}(\bar\Omega)$ for any $\alpha\in (0,1)$, but $u\notin C^{s}(\bar\Omega)$.
Proceeding similarly as in Case 1, we can raise the regularity of $f$ but maintain the regularity of $u$. 

\smallskip 
{\it Case 3: Singular monomial and logarithmic factors.} 
We assume that $s$ is not an integer and reexamine Case 2.
With $s=k+\alpha$, for some integer $k\ge 0$ and constant $\alpha\in (0,1)$,
and $u$ and $f$ as in Case 2, we have $f\in C^{k,\alpha}(\bar\Omega)$
and $u\in C^{k,\beta}(\bar\Omega)$ for any $\beta\in (0,\alpha)$, 
but $u\notin C^{k,\alpha}(\bar\Omega)$. Note that 
$f\notin C^{k,\gamma}(\bar\Omega)$ for any $\gamma\in (\alpha,1)$, if $a(2s-1)+b\neq0$.
\end{example} 

In Case 1 and Case 2, the regularity of the nonhomogeneous terms $f$
can be improved to orders as large as we desire. However, in Case 3, 
the regularity of $f$ is kept at a certain order.

In view of Example \ref{exa-ch2-solutions-not-regular}, we introduce two important concepts, 
characteristic polynomials and characteristic exponents.  
Take an arbitrary $\mu\in\mathbb R$. A straightforward computation yields 
$$\rho^{-\mu}L\rho^\mu=\mu(\mu-1)a_{ij}\partial_i\rho\partial_j\rho+\mu b_i\partial_i\rho+c
+\mu\rho a_{ij}\partial_{ij}\rho.$$
Note that $\nabla \rho|_{\partial\Omega}$ is the inner unit normal 
vector along $\partial\Omega$. 
Under the extra condition that $\rho D^2\rho\in C(\bar\Omega)$ with 
$\rho D^2\rho=0$ on $\partial\Omega$, we have 
$$\rho^{-\mu}L\rho^\mu\big|_{\partial\Omega}=\mu(\mu-1)a_{ij}\nu_i\nu_j
+\mu b_i\nu_i+c,$$
where $\nu=(\nu_1, \cdots, \nu_n)$ is the inner unit normal along $\partial\Omega$. 
We now define 
\begin{align}\label{eq-ch2-definition-P-mu}
P(\mu)=\mu(\mu-1)a_{ij}\nu_i\nu_j
+\mu b_i\nu_i+c
\quad\text{on }\partial\Omega.\end{align}
For each fixed point on $\partial\Omega$, $P(\mu)$ is a quadratic polynomial in $\mu$ 
with a positive leading coefficient. The polynomial $P(\mu)$ is referred to as the 
{\it characteristic polynomial} of $L$ and its roots as {\it characteristic exponents}
or {\it indicial roots}. 

For the operator $L$ in \eqref{eq-ch2-def-operator-example}, 
we have $a_{ij}=a\delta_{ij}$ and $b_i=b\partial_i\rho$ if we put it in the form \eqref{eq-ch2-def-operator}. 
Hence, the characteristic polynomial of $L$ is given by 
$$P(\mu)=a\mu(\mu-1)+b\mu +c.$$
By \eqref{eq-ch2-root-condition}, $s$ is a positive root of $P(\mu)$. 

Example \ref{exa-ch2-solutions-not-regular} demonstrates that positive characteristic exponents serve as an obstruction 
to the boundary regularity of solutions $u$, regardless of the regularity of $f$. 
Different types of singularity arise depending on whether the positive characteristic exponents are integers. 
If they are integers, solutions have logarithmic factors. 
If they are not integers, solutions have forms of monomials of non-integer powers 
or monomials of non-integer powers coupled with logarithmic factors, 
depending on the regularity of nonhomogeneous terms.

Now, we introduce some conditions on the boundary to offset the effect of the degeneracy of the operator. First, 
we assume 
$$P(0)=c<0\quad\text{on }\partial\Omega.$$ 
Hence, the characteristic polynomial $P(\mu)$ has two real roots, a positive one and a negative one,
and any $\mu$ with a negative value of $P(\mu)$ 
has to be between these two roots. For some integer $k\ge 0$ and constant $\alpha\in (0,1)$, a general regularity 
of $C^{k,\alpha}$ for solutions is possible only under the condition 
$P({k+\alpha})<0$ on $\partial\Omega$.
This is clear from Example \ref{exa-ch2-solutions-not-regular}.

Let $f$ be a continuous function in $\bar\Omega$. Consider the Dirichlet problem given by 
\eqref{eq-ch2-basic-equation} and 
\begin{align}
\label{eq-ch2-Dirichlet}
u=\frac{f}{c}\quad\text{on }\partial\Omega. 
\end{align}
The boundary value $f/c$ in \eqref{eq-ch2-Dirichlet} is not arbitrarily prescribed and 
is determined by the equation itself, specifically by the ratio of $f$ and $c$. This is due to the degeneracy along the 
boundary. In fact, normal derivatives of solutions on the boundary up to a certain order are also determined by the equation. 

Our main focus in this paper is the regularity of solutions 
of the Dirichlet problem \eqref{eq-ch2-basic-equation} and \eqref{eq-ch2-Dirichlet}. 
If $a_{ij}$, $b_i$, $c$, and $f$ are $C^{k,\alpha}(\bar\Omega)$-functions, 
for some integer $k\ge 0$ and some 
constant $\alpha\in (0,1)$, then the solution $u$ of \eqref{eq-ch2-basic-equation} and \eqref{eq-ch2-Dirichlet}
is a $C^{k+2,\alpha}(\Omega)$-function. This follows from the interior Schauder theory. 
However, the global Schauder theory cannot be applied since the operator $L$ is 
degenerate along the boundary. 
The main result in this paper concerns the regularity of solutions up to the boundary.

\begin{theorem}\label{thrm-ch3-regularity-higher}
For some integer $k\ge 0$ and constant $\alpha\in (0,1)$, 
let $\Omega$ be a bounded domain in $\mathbb R^n$
with a $C^{k+1,\alpha}$ boundary $\partial\Omega$ and 
$\rho$ be a $C^{k+1, \alpha}(\bar\Omega)$-defining function with 
$\rho\nabla^{k+2}\rho\in C^{\alpha}(\bar\Omega)$ and
$\rho\nabla^{k+2}\rho=0$ on $\partial\Omega$. 
Assume $a_{ij}, b_i, c\in C^{k,\alpha}(\bar\Omega)$, 
with \eqref{eq-ch2-ellipticity}, 
and $c\le -c_0$ and $P({k+\alpha})\le -c_{k+\alpha}$ on $\partial\Omega$, 
for some positive constants $c_0$ and $c_{k+\alpha}$. 
For some $f\in C^{k,\alpha}(\bar\Omega)$, 
suppose $u\in C(\bar\Omega)\cap C^{2}(\Omega)$
is a solution of the Dirichlet problem \eqref{eq-ch2-basic-equation} and \eqref{eq-ch2-Dirichlet}.
Then, $u\in C^{k,\alpha}(\bar\Omega)\cap C^{k+2,\alpha}(\Omega)$, 
$\rho\nabla^{k+1} u, \rho^2\nabla^{k+2}u\in C^{\alpha}(\bar\Omega)$, 
and $\rho\nabla^{k+1} u=0$ and $\rho^2\nabla^{k+2}u=0$ on $\partial\Omega$. Moreover, 
\begin{align}\label{eq-ch3-estimate-solution-Holder-higher}
|u|_{C^{k,\alpha}(\bar\Omega)}
+|\rho\nabla^{k+1}u|_{C^{\alpha}(\bar\Omega)}+ |\rho^2\nabla^{k+2}u|_{C^{\alpha}(\bar\Omega)}
\le C\big\{|u|_{L^\infty(\Omega)}+|f|_{C^{k,\alpha}(\bar\Omega)}\big\},
\end{align}
where  $C$ is a positive constant depending only on $n$, $\lambda$, 
$k$, $\alpha$, $c_0$, $c_{k+\alpha}$, $\Omega$, 
the $C^{k+1,\alpha}$-norm of $\rho$ in $\bar\Omega$, 
the $C^{\alpha}$-norm of $\rho\nabla^{k+2}\rho$ in $\bar\Omega$, 
and the $C^{k,\alpha}$-norms of $a_{ij}$, $b_i$, and $c$ in $\bar\Omega$. 
\end{theorem}

Although formulated as a global result, 
Theorem \ref{thrm-ch3-regularity-higher}
actually holds locally near the boundary. 
We now explain some assumptions in Theorem \ref{thrm-ch3-regularity-higher}. 
By the assumption  
$P(0)=c<0$ on $\partial\Omega$, 
the characteristic polynomial $P(\mu)$ has two real roots, a positive one and a negative one,
and any $\mu$ with a negative value of $P(\mu)$ 
has to be between these two roots. For some integer $k\ge 0$ and constant $\alpha\in (0,1)$, 
the assumption $P({k+\alpha})<0$ on $\partial\Omega$ implies that $k+\alpha$ is less than the positive root
of $P(\mu)$. 
We emphasize that the only assumption on $f$ is on its regularity, i.e., $f\in C^{k,\alpha}(\bar\Omega)$. 
There are no additional assumptions on its growth or decay near the boundary. 

As a consequence of Theorem \ref{thrm-ch3-regularity-higher} with $k=0$ and standard regularization, 
we obtain the following existence result.

\begin{theorem}\label{thrm-ch2-Existence-Dirichlet}
For some constant $\alpha\in (0,1)$, 
let $\Omega$ be a bounded domain in $\mathbb R^n$
with a $C^{1,\alpha}$ boundary $\partial\Omega$ and 
$\rho$ be a $C^{1, \alpha}(\bar\Omega)$-defining function with 
$\rho\nabla^{2}\rho\in C^{\alpha}(\bar\Omega)$ and
$\rho\nabla^{2}\rho=0$ on $\partial\Omega$. 
Assume $a_{ij}, b_i, c\in C^{\alpha}(\bar\Omega)$,  
with \eqref{eq-ch2-ellipticity}, $c\le 0$ in $\Omega$, 
and $c<0$ and $P(\alpha)<0$ on $\partial\Omega$.  
Then, for any  $f\in C^\alpha(\bar\Omega)$, 
the Dirichlet problem \eqref{eq-ch2-basic-equation} and \eqref{eq-ch2-Dirichlet} 
admits a unique solution 
$u\in C^\alpha(\bar\Omega)\cap 
C^{2,\alpha}(\Omega)$, with  $\rho \nabla u$, $\rho^2\nabla^2u\in C^{\alpha}(\bar\Omega)$, 
and $\rho \nabla u=0$ and $\rho^2\nabla^2u=0$ on $\partial\Omega$. 
\end{theorem}

A new component of Theorem \ref{thrm-ch2-Existence-Dirichlet}
is the estimate of the $L^\infty$-norm of $u$ under the additional assumption 
$c\le 0$ in $\Omega$.
In view of Theorem \ref{thrm-ch3-regularity-higher}, for $k\ge 0$ and $\alpha\in (0,1)$, we set 
\begin{align}\label{eq-ch3-Holder-space}
C^{k, \alpha}_2(\bar\Omega)=\{u\in C^{k,\alpha}(\bar\Omega)\cap C^{k+2,\alpha}(\Omega): 
\rho^i\nabla^{k+i} u\in C^\alpha(\bar\Omega)\text{ for }i=1, 2\},
\end{align}
and equip it with the norm 
$$|u|_{{C}^{k, \alpha}_2(\bar\Omega)}=|u|_{C^{k,\alpha}(\bar\Omega)}
+\sum_{i=1}^2 |\rho^i\nabla^{k+i}u|_{C^{\alpha}(\bar\Omega)}.$$
Then, 
$({C}^{k, \alpha}_2(\bar\Omega), |\cdot|_{{C}^{k, \alpha}_2(\bar\Omega)})$
is a Banach space. 
Theorem \ref{thrm-ch2-Existence-Dirichlet} and Theorem \ref{thrm-ch3-regularity-higher} 
assert that $L: {C}^{k, \alpha}_2(\bar\Omega)\to C^{k,\alpha}(\bar\Omega)$ is an isomorphism, 
if $c\le 0$ in $\Omega$, and $c<0$ and $P({k+\alpha})<0$ on $\partial\Omega$. 
We point out that the target space in the isomorphism above is the usual H\"older space. 
Moreover, it is a general fact that  $\rho^i\nabla^{k+i}u=0$ 
on $\partial\Omega$, $i=1, 2$,  for $u\in C^{k, \alpha}_2(\bar\Omega)$.

We next describe our approach. 
The essential step in proving Theorem \ref{thrm-ch3-regularity-higher}
is to establish decay estimates of solutions near the boundary.
With the help of a calculus lemma, such decay estimates improve 
interior estimates of the H\"older semi-norms to 
global estimates of the H\"older semi-norms. 
Central to this approach is to construct appropriate homogenous supersolutions. 
With these supersolutions, we prove separate decays of tangential derivatives and 
normal derivatives. 
The study of the normal derivatives is much more involved. To discuss derivatives along the normal direction, 
we first establish a higher-order decay estimate of solutions  
to conclude the existence of normal derivatives on the boundary and then establish a decay estimate of 
normal derivatives to get the desired regularity of derivatives along the normal direction. 

For uniformly elliptic equations, 
it suffices to discuss the regularity of solutions along the tangential directions, 
since derivatives along the normal direction can be expressed 
in terms of other derivatives by the equation. 
However, this is not true anymore for equations that are degenerate along the boundary. 
In fact, regularity along the normal direction is a delicate issue for 
uniformly degenerate elliptic equations. 
In this paper, we use elliptic PDE methods for both tangential derivatives and normal derivatives. 
This is different from arguments based on ODEs for normal derivatives as in \cite{Lin1989Invent}. 
One advantage of the pure PDE approach  
is that the positive root of the characteristic polynomial is not assumed to be constant, 
although it is 
for many geometric problems.  

This paper is organized as follows. 
In Section \ref{sec-decay-estimates-linear}, 
we reformulate the boundary value problem 
\eqref{eq-ch2-basic-equation} and \eqref{eq-ch2-Dirichlet} 
near a flat portion of the boundary, 
and construct a class of homogeneous 
supersolutions and derive decay estimates of solutions. 
In Section \ref{sec-OptimalRegularity-linear-general-case}, we study the H\"older 
regularity of solutions up to the boundary. 
In Section \ref{sec-OptimalRegularity-linear-tangential-regularity}, 
we study the regularity of solutions 
along tangential directions. 
We prove that 
solutions are as regular as allowed 
by the coefficients and nonhomogeneous terms along the tangential directions.
In Section \ref{sec-OptimalRegularity-normal-regularity}, 
we discuss the regularity along the normal direction. 
We prove that 
solutions are regular along the normal direction up to a certain order.
In Section \ref{sec-existence-linear-existence}, we outline a proof 
of Theorem \ref{thrm-ch2-Existence-Dirichlet}
by regularizing the uniformly degenerate elliptic equations.

\section{Decay Estimates}\label{sec-decay-estimates-linear} 

In Sections \ref{sec-decay-estimates-linear}-\ref{sec-OptimalRegularity-normal-regularity}, 
we prove a local version of 
Theorem \ref{thrm-ch3-regularity-higher}. 
By an appropriate change of coordinates, we assume 
that the domain $\Omega$ has a portion of flat boundary and 
that the defining function is given by the distance function to the boundary. 
In these sections, summations over Latin letters are from 1 to $n$ 
and those over Greek letters are from 1 to $n-1$. 
However, an unrepeated $\alpha\in (0,1)$ is reserved for 
the H\"older index in the study of regularity for elliptic equations.

Denote by $x=(x',t)$ points in $\mathbb{R}^{n}$ and also write $x_{n}=t$. 
Set, for any $x_0'\in\mathbb R^{n-1}$ and any $r>0$, 
\begin{align*}
B'_{r}(x_0')=\{x'\in\mathbb{R}^{n-1}:|x'-x_0'|<r\},
\end{align*}
and 
\begin{align*}
&G_{r}(x_0')=\{(x',t)\in\mathbb{R}^{n}:|x'-x_0'|<r,\ 0<t<r\},\\
&\Sigma_{r}(x_0')=\{(x',0)\in\mathbb{R}^{n}:|x'-x_0'|<r\}. 
\end{align*}
We also set $B'_{r}=B'_{r}(0)$, $G_{r}=G_{r}(0)$, and $\Sigma_{r}=\Sigma_{r}(0)$.

Consider the operator
\begin{align}\label{eq-ch3-LinearOperator-t}
L=t^2a_{ij}\partial_{ij}+tb_i\partial_i+c\quad\text{in }G_1,\end{align}
with $a_{ij}, b_i, c\in C(\bar G_1)$ 
satisfying $a_{ij}=a_{ji}$ and, 
for any $x\in G_1$ and $\xi\in\mathbb R^n$, 
\begin{align}\label{eq-ch3-ellipticity-t}
\lambda|\xi|^2\le a_{ij}(x)\xi_i\xi_j\le\Lambda|\xi|^2,\end{align}
for some positive constants $\lambda$ and $\Lambda$. 
Similarly as in \eqref{eq-ch2-definition-P-mu}, define $Q(\mu)$ by 
$$Q(\mu)=\mu(\mu-1)a_{nn}+\mu b_n+c\quad\text{in }G_1.$$
We point out that $Q(\mu)$ here is defined in the entire domain $G_1$, instead of $P(\mu)$ only on the boundary. 
Note that $Q(0)=c$. 
We always assume
$c\le -c_0$ in $G_1,$
for some positive constant $c_0$.

We consider 
\begin{align}\label{eq-ch3-Equ} Lu= f\quad\text{in }G_1,
\end{align}
and 
\begin{align}\label{eq-ch3-Dirichlet}
u(\cdot, 0)=u_0\quad\text{on }B'_1,\end{align}
where we always assume 
\begin{align}\label{eq-ch3-boundary-value}u_0=\frac{f}{c}(\cdot, 0)
\quad\text{on }B'_1.\end{align}
We point out that the assumption $c<0$ on $\Sigma_1$ is essential in order to relate 
the boundary value of $u$ to values of the nonhomogeneous term $f$.

For some constant $\mu>0$, a simple computation yields 
$$Lt^\mu=Q(\mu) t^\mu.$$
If $Q(\mu)\le -c_\mu$ in $G_1$ 
for some positive constant $c_\mu$, 
then
$Lt^\mu\le -c_\mu t^\mu$  in $G_1$. 
In particular, $t^\mu$ is a supersolution. 
Next, by modifying $t^\mu$, we construct a class of homogeneous 
supersolutions,  
which play an essential role in our discussion and enable us to derive decay estimates of solutions.

\begin{lemma}\label{lemma-ch3-Linear-ComparisonFunction-t-factor} 
Let $\sigma$ and $\mu$ be constants with $0\le \sigma<\mu$. 
Assume $a_{ij}, b_i, c\in C(\bar G_1)$, 
with \eqref{eq-ch3-ellipticity-t},  and
$Q(\sigma)\le -c_\sigma$ and $Q(\mu)\le -c_\mu$ in $G_1$, for some 
positive constants $c_\sigma$ and $c_\mu$.
Then,  for some nonnegative constants $\varepsilon$ and $K$, depending only on $\mu$, $c_\mu$, and 
the $L^\infty$-norms 
of $a_{ij}$ and $b_i$ in $G_1$, 
\begin{align}\label{eq-ch3-supersolutions-sigma-factor}L\big[t^\sigma(\varepsilon |x'|^2+t^2)^{\frac{\mu-\sigma}2}
+K t^\mu\big]
\le -\frac12c_\sigma t^\sigma(\varepsilon |x'|^2+t^2)^{\frac{\mu-\sigma}2}\quad\text{in }G_1.\end{align}
\end{lemma}

\begin{proof} 
Set, for some $\varepsilon\in (0,1)$
to be determined,   
\begin{equation}\label{eq-ch3-SuperDegree-mu}
\widehat \psi(x)=t^\sigma(\varepsilon|x'|^2+t^2)^{\frac{\mu-\sigma}2}.\end{equation}
Write $L=L_1+L_2,$
with
\begin{align*}L_1&=t^2a_{nn}\partial_{tt}+tb_n\partial_t+c,\\
L_2&=2t^2a_{\alpha n}\partial_{\alpha t}
+t^2a_{\alpha\beta}\partial_{\alpha\beta}+tb_\alpha\partial_\alpha.\end{align*}
A straightforward calculation yields
$$L_i\widehat \psi=t^{\sigma}(\varepsilon|x'|^2+t^2)^{\frac{\mu-\sigma}2-2}I_i\quad\text{for }i=1,2,$$
where
\begin{align}\label{eq-ch3-expression-I1}\begin{split}
I_1&=(\mu-\sigma)(\mu-\sigma-2)a_{nn}t^4+(\sigma(\sigma-1)a_{nn}+\sigma b_n+c)(\varepsilon|x'|^2+t^2)^2\\
&\qquad+(\mu-\sigma) ((2\sigma+1)a_{nn}+ b_n)t^2(\varepsilon|x'|^2+t^2),
\end{split}\end{align}
and 
\begin{align*}
I_2&=(\mu-\sigma)(\mu-\sigma-2)t^2
(2\varepsilon a_{\alpha n}x_\alpha t+\varepsilon^2a_{\alpha\beta}x_\alpha x_\beta)\\
&\qquad+(\mu-\sigma)(\varepsilon|x'|^2+t^2)(\varepsilon a_{\alpha\beta}\delta_{\alpha\beta} t^2
+\varepsilon(2\sigma a_{\alpha n}+b_\alpha) t x_\alpha).
\end{align*}
By $2\sqrt{\varepsilon}|x'|t\le \varepsilon|x'|^2+t^2$, we obtain 
$$I_2\le C_1\sqrt{\varepsilon}(\varepsilon|x'|^2+t^2)^2,$$
and hence
\begin{equation}\label{eq-ch3-Estimate-hat-w1}
L_2\widehat \psi\le C_1\sqrt{\varepsilon}t^\sigma(\varepsilon|x'|^2+t^2)^{\frac{\mu-\sigma}2},
\end{equation}
where $C_1$ is a positive constant depending only on $\mu$ and 
the $L^\infty$-norms of $a_{ij}$ and $b_i$ in $G_1$. 
By  \eqref{eq-ch3-expression-I1} and $Q(\sigma)\le -c_\sigma$, we have
\begin{align*}
I_1 &\le
-c_\sigma(\varepsilon|x'|^2+t^2)^2+(\mu-\sigma)(\mu-\sigma-2)a_{nn}t^4\\
&\qquad+(\mu-\sigma) ((2\sigma+1)a_{nn}+ b_n)t^2(\varepsilon|x'|^2+t^2).
\end{align*}
By the Cauchy inequality, we get 
\begin{align*}I_1\le 
-\frac78c_\sigma (\varepsilon|x'|^2+t^2)^2+C_2t^4,\end{align*} 
where $C_2$ is a positive constant depending only on $\mu$ 
and the $L^\infty$-norms of $a_{nn}$ and
$b_{n}$ in $G_1$. Hence, 
\begin{equation}\label{eq-ch3-Estimate-hat-w2}
L_1\widehat \psi\le t^{\sigma}(\varepsilon|x'|^2+t^2)^{\frac{\mu-\sigma}2}
\Big\{-\frac78c_\sigma +\frac{C_2t^4}{(\varepsilon|x'|^2+t^2)^2}\Big\}.\end{equation}
By adding \eqref{eq-ch3-Estimate-hat-w1} and \eqref{eq-ch3-Estimate-hat-w2}, we get 
$$L\widehat \psi\le t^\sigma(\varepsilon|x'|^2+t^2)^{\frac{\mu-\sigma}2}\Big\{-\frac78c_\sigma+C_1\sqrt{\varepsilon}
+\frac{C_2t^4}{(\varepsilon|x'|^2+t^2)^2}\Big\}.$$
By taking $\varepsilon$ small, we have
\begin{equation}\label{eq-ch3-Estimate-hat-w3}
L\widehat \psi\le t^\sigma(\varepsilon|x'|^2+t^2)^{\frac{\mu-\sigma}2}\Big\{-\frac34c_\sigma
+\frac{C_2t^4}{(\varepsilon|x'|^2+t^2)^2}\Big\}.\end{equation}
In the following, we fix such an $\varepsilon$. 
Next, we set
\begin{equation}\label{eq-ch3-SuperDegree-mu-modified}
\widetilde{\psi}=t^\mu.\end{equation}
By $Q(\mu)\le -c_\mu$,  we have 
\begin{equation}\label{eq-ch3-Estimate-tilde-w1}
L\widetilde \psi=Q(\mu) t^\mu\le -c_\mu t^\mu.\end{equation}

We now combine \eqref{eq-ch3-Estimate-hat-w3} and \eqref{eq-ch3-Estimate-tilde-w1}
and consider $L(\widehat \psi+K\widetilde \psi)$, for some nonnegative constant $K$. 
Take a constant $\delta>0$ to be determined. 
We first consider $t\le \delta  (\varepsilon|x'|^2+t^2)^{1/2}$. By \eqref{eq-ch3-Estimate-hat-w3}, we get 
$$L\widehat \psi\le t^\sigma(\varepsilon|x'|^2+t^2)^{\frac{\mu-\sigma}2}\Big\{-\frac34c_\sigma
+C_2\delta^4\Big\}
\le -\frac12c_\sigma t^\sigma(\varepsilon|x'|^2+t^2)^{\frac{\mu-\sigma}2},$$
by taking $\delta$ small. 
Note that \eqref{eq-ch3-Estimate-tilde-w1} implies
$L\widetilde \psi\le0.$
Therefore, for any $K\ge 0$ and $t\le \delta  (\varepsilon|x'|^2+t^2)^{1/2}$, 
\begin{equation}\label{eq-ch3-SuperDegree-mu1}
L(\widehat \psi+K\widetilde \psi)\le -\frac12c_\sigma t^\sigma(\varepsilon|x'|^2+t^2)^{\frac{\mu-\sigma}2}.\end{equation}
Next, we consider $t\ge\delta (\varepsilon|x'|^2+t^2)^{1/2}$. By \eqref{eq-ch3-Estimate-tilde-w1}, 
we have
$$
L\widetilde \psi\le -c_\mu\delta^{\mu-\sigma}t^\sigma (\varepsilon|x'|^2+t^2)^{\frac{\mu-\sigma}2}.
$$
By \eqref{eq-ch3-Estimate-hat-w3}, we get
$$
L\widehat \psi\le C_2 t^\sigma(\varepsilon|x'|^2+t^2)^{\frac{\mu-\sigma}2}.$$
Hence, 
$$L(\widehat \psi+K\widetilde \psi)\le t^\sigma(\varepsilon|x'|^2+t^2)^{\frac{\mu-\sigma}2}
\big(C_2-Kc_\mu\delta^{\mu-\sigma}\big).$$
By taking $K$ sufficiently large, 
we obtain \eqref{eq-ch3-SuperDegree-mu1} for $t\ge\delta (\varepsilon|x'|^2+t^2)^{1/2}$.
Therefore,  \eqref{eq-ch3-SuperDegree-mu1} holds in $G_1$. 
We have the desired result by noting that $\widehat \psi$ and $\widetilde \psi$ are given by 
\eqref{eq-ch3-SuperDegree-mu} and \eqref{eq-ch3-SuperDegree-mu-modified}, respectively. 
\end{proof}

With Lemma \ref{lemma-ch3-Linear-ComparisonFunction-t-factor}, we now 
derive a decay estimate near $\Sigma_1$. 

\begin{lemma}\label{lemma-ch3-Linear-Estimate-alpha}
For some constant 
$\alpha\in (0,1)$,  
assume $a_{ij}, b_i\in C(\bar G_1)$ and $c\in C^{\alpha}(\bar G_1)$, 
with \eqref{eq-ch3-ellipticity-t}, and $c\le -c_0$ 
and $Q({\alpha})\le -c_{\alpha}$  in $G_1$, for some positive constants 
$c_0$ and $c_{\alpha}$.
For some $f\in C^\alpha(\bar G_1)$, let $u\in L^\infty(G_1)\cap C^2(G_1)$ be 
a solution of \eqref{eq-ch3-Equ}.   
Then, for any $r\in (0,1)$,  any 
$x_0'\in B_{r}'$, and any $(x',t)\in G_1$, 
\be\label{eq-ch3-linear-estimate-alpha}
|u(x',t)-u_0(x'_0)|
\le C(|x'-x_0'|^2+t^2)^{\frac\alpha2}\big\{|u|_{L^\infty(G_1)}+|f|_{C^\alpha(\bar G_1)}\big\},\ee
where $u_0$ is given by \eqref{eq-ch3-boundary-value}
and $C$ is a positive constant depending only on $n$, $r$, $\lambda$, $c_0$, $c_{\alpha}$, 
the $L^\infty$-norms of $a_{ij}, b_i$ in $G_1$, 
and the $C^\alpha$-norm of $c$ in $\bar G_1$.
\end{lemma}

\begin{proof} Set 
$$F=|u|_{L^\infty(G_1)}+|f|_{C^\alpha(\bar G_1)}.$$
Fix an $r\in (0,1)$ and take any $x_0'\in B'_r$. A simple subtraction yields 
\begin{equation}\label{eq-ch3-Equ-alpha-general} L(u-u_0(x'_0))= h\quad\text{in }G_1,\end{equation}
where
$$h=f-u_0(x'_0) c=f-f(x_0',0)-u_0(x'_0)(c-c(x'_0,0)).$$
Here, we used $f(x'_0,0)= c(x'_0,0)u_0(x'_0)$
by \eqref{eq-ch3-boundary-value}. Hence, 
\begin{align}\label{eq-boundary-estimate-h}
\pm h\ge -C|f|_{C^\alpha(\bar G_1)}(|x'-x'_0|^2+t^2)^{\frac\alpha2}\quad\text{in }G_1.\end{align}

We first consider a special case that $u\in C(\bar G_1)$
and $u(\cdot, 0)=u_0$ on $B_1'$. 
Then,
$$\pm (u(x',0)-u_0(x'_0))=\pm\Big(\frac{f}{c}(x',0)-\frac{f}{c}(x'_0,0)\Big)
\le C|f|_{C^\alpha(\bar G_1)}|x'-x'_0|^\alpha\quad\text{on }\Sigma_{1},$$
and
$$\pm (u(x',t)-u_0(x_0'))\le 2|u|_{L^\infty(B_1)}\quad\text{in } G_{1}.$$
Set
\begin{equation}\label{eq-ch3-ComparisonFunction}
\psi(x',t)=(\varepsilon |x'-x'_0|^2+t^2)^{\frac\alpha2}
+K t^{\alpha}.
\end{equation}
By Lemma \ref{lemma-ch3-Linear-ComparisonFunction-t-factor}
with $\sigma=0$ and $\mu=\alpha$, we have, for some positive constants 
$\varepsilon$ and $K$, 
$$L\psi\le -\frac12c_0(\varepsilon |x'-x'_0|^2+t^2)^{\frac\alpha2}.$$
Moreover, 
$$\psi(x',0)= \varepsilon^{\frac\alpha2}|x'-x'_0|^\alpha\quad\text{on }\Sigma_{1},$$
and
$$\psi \ge \varepsilon^{\frac\alpha2}(1-r)^\alpha\quad\text{on }\partial G_1\setminus \Sigma_{1}.$$
For some constant $A$ sufficiently large, we obtain 
$$\aligned \pm L(u-u_0(x'_0))&\ge L(AF\psi)\quad\text{in }G_1,\\ 
\pm(u-u_0(x'_0))&\le AF\psi\quad\text{on }\partial G_1.\endaligned$$
The maximum principle implies 
$$\pm (u-u_0(x'_0))\le AF\psi\quad\text{in }G_1,$$
and hence, for any $(x',t)\in G_{1}$, 
$$|u(x', t)-u_0(x'_0)|\le A(1+K)F(|x'-x'_0|^2+t^2)^{\frac\alpha2}.$$ 
This is the desired estimate.

We now consider the general case that $u\in L^\infty(G_1)$. 
Take any $\kappa\in (0,1)$ and set 
$$v=t^\kappa(u-u_0(x'_0))=\frac{u-u_0(x'_0)}{t^{-\kappa}}.$$
Then, $v$ is continuous up to $\Sigma_1$ with $v=0$ on $\Sigma_1$, and 
$$\pm v\le t^\kappa \big\{|u|_{L^\infty(G_1)}+|u_0|_{L^\infty(B_1')}\big\}
\le C t^\kappa \big\{|u|_{L^\infty(G_1)}+|f|_{L^\infty(G_1)}\big\}.$$
By \eqref{eq-ch3-Equ-alpha-general} and a straightforward computation, we have 
$$L_{-\kappa}v=t^\kappa h,$$ 
where 
\begin{equation*}
L_{-\kappa}=t^2a_{ij}\partial_{ij}+t(b_i-2\kappa a_{in})\partial_i+Q(-\kappa).
\end{equation*}
Note that $L_{-\kappa}$ has a similar structure as $L$ and has an associated 
$Q^{(-\kappa)}(\mu)$ given by 
$$Q^{(-\kappa)}(\mu)=\mu(\mu-1)a_{nn}+\mu (b_n-2\kappa a_{nn})+Q(-\kappa).$$
A simple computation yields 
$Q^{(-\kappa)}(\mu)=Q({\mu-\kappa})$ in $Q_1.$
In particular, we have $Q^{(-\kappa)}({\kappa})$ $=Q(0)=c\le -c_0$ and 
$Q^{(-\kappa)}({\kappa+\alpha})=Q(\alpha)\le -c_\alpha$ in $Q_1.$ 
Instead of \eqref{eq-ch3-ComparisonFunction}, set 
\begin{equation*}\psi(x',t)=t^{\kappa}(\varepsilon |x'-x'_0|^2+t^2)^{\frac\alpha2}
+K t^{\alpha+\kappa}.
\end{equation*}
By applying Lemma \ref{lemma-ch3-Linear-ComparisonFunction-t-factor} to $L_{-\kappa}$ with 
$\sigma=\kappa$ and $\mu=\kappa+\alpha$, 
we have
$$L_{-\kappa}\psi\le -\frac12c_0 t^{\kappa}(\varepsilon |x'-x'_0|^2+t^2)^{\frac{\alpha}2},$$
for some positive constants 
$\varepsilon$ and $K$.
By \eqref{eq-boundary-estimate-h}, we get
$$\pm t^\kappa h\ge -Ct^\kappa |f|_{C^\alpha(\bar G_1)}(|x'-x'_0|^2+t^2)^{\frac\alpha2}\quad\text{in }G_1.$$
By proceeding as in the proof in the special case, we have, 
for any $(x',t)\in G_{1}$, 
$$|v(x', t)|\le A(1+K)t^\kappa F(|x'-x'_0|^2+t^2)^{\frac\alpha2},$$ 
and hence
$$|u(x', t)-u_0(x'_0)|\le A(1+K)F(|x'-x'_0|^2+t^2)^{\frac\alpha2}.$$ 
This is the desired estimate. 
We note that the constants $\varepsilon$ and $K$ in the definition of $\psi$ and $A$ above 
are in fact independent of $\kappa\in (0,1)$. 
\end{proof}

We now make several remarks. 
First, we only assume that the solution $u$ is bounded in $G_1$. The estimate \eqref{eq-ch3-linear-estimate-alpha}
implies in particular that $u$ is continuous up to $\Sigma_1$ with $u=u_0$ on $\Sigma_1$. 
This is an important feature of Lemma \ref{lemma-ch3-Linear-Estimate-alpha}. 
Second, for a fixed $x_0'\in B_{1/2}'$, the estimate \eqref{eq-ch3-linear-estimate-alpha}
actually establishes that $u$ is $C^\alpha$ at $(x_0',0)$. In fact, it holds under a weaker assumption 
that $f$ and $c$ are $C^\alpha$ at $(x_0',0)$. 
Third, we can prove \eqref{eq-ch3-linear-estimate-alpha} alternatively by a maximum principle due to 
Cheng and Yau \cite{ChengYau1975}. See also \cite{GrahamLee1991}. 

\section{The H\"older Regularity}\label{sec-OptimalRegularity-linear-general-case}

In this section, we study the H\"older 
regularity of solutions up to the boundary. 
We first state a simple lemma. 

\begin{lemma}\label{lemma-ch3-GlobalHolder} Let $A>0$ and $\alpha\in (0,1)$ be constants, 
and $w\in C^\alpha(G_1)$ and $w_0\in C^\alpha(B_1')$ be functions. Suppose, 
for any $(x',t)\in G_{1}$ with $B_{t/2}(x',t)$ $\subset G_1$, 
\begin{equation}\label{eq-ch3-GlobalHolder1}|w(x',t)-w_0(x')|\le At^\alpha,\end{equation} 
and 
\begin{equation}\label{eq-ch3-GlobalHolder2}[w]_{C^\alpha(B_{t/2}(x',t))}\le A.\end{equation}
Then, $w\in C^\alpha(\bar G_{r})$, for any $r\in (0,1)$, and 
$$|w|_{C^\alpha(\bar G_{1/2})}\le 5A+|w_0|_{C^\alpha(B_1')}.$$
\end{lemma} 

The proof is standard and hence omitted. 
Compare with Propositions 4.12-4.13 \cite{Caffarelli-Cabre1995}.

We now prove a general H\"older regularity for uniformly degenerate elliptic equations 
that solutions with appropriate decays near the boundary are H\"older continuous up to the boundary. 
We introduce the scaled H\"older norm on balls for brevity. 
For some ball $B_r(x)\subset\mathbb R^n$ and some H\"older continuous function $w\in C^\alpha(\bar B_r(x))$, 
define 
\begin{equation}\label{eq-ch2-scaled-Holder-norm}
|w|^*_{C^\alpha(B_r(x))}=|w|_{L^\infty(B_r(x))}+r^\alpha[w]_{C^\alpha(B_r(x))}.\end{equation}

\begin{lemma}\label{lemma-ch3-Linear-Regularity-general}
For some constant 
$\alpha\in (0,1)$, assume $a_{ij}, b_i, c\in C^{\alpha}(\bar G_1)$, 
with \eqref{eq-ch3-ellipticity-t}, and $c\le -c_0$  in $G_1$, for some positive constant
$c_0$.
For some $f\in C^\alpha(\bar G_1)$, let $u\in C(\bar G_1)\cap C^2(G_1)$ be 
a solution of \eqref{eq-ch3-Equ} and \eqref{eq-ch3-Dirichlet}, satisfying, for any $(x',t)\in G_1$, 
\begin{equation}\label{eq-ch3-decay-assumption-general}
|u(x',t)-u_0(x')|\le A_0t^\alpha,
\end{equation}
for the function $u_0$ given by \eqref{eq-ch3-boundary-value} and some positive constant $A_0$. 
Then, for any $r\in (0,1)$, 
\begin{equation*}
u, tDu, t^2D^2u\in C^\alpha(\bar G_r),\end{equation*}
with $tDu=0$ and $t^2D^2u=0$ on $\Sigma_1$, and 
\begin{align*}
|u|_{C^\alpha(\bar G_{1/2})}+|tDu|_{C^\alpha(\bar G_{1/2})}
+|t^2D^2u|_{C^\alpha(\bar G_{1/2})}
\le C\big\{A_0+|u|_{L^\infty(G_1)}+|f|_{C^\alpha(\bar G_1)}\big\},
\end{align*}
where $C$ is a positive constant depending only on $n$, $\lambda$, 
$\alpha$,  $c_0$, 
and
the $C^\alpha$-norms of $a_{ij}, b_i, c$ in $\bar G_1$. 
\end{lemma}

Note that \eqref{eq-ch3-decay-assumption-general} implies, for any $(x',t)\in G_1$ and any $x_0'\in B_1'$, 
\begin{equation}\label{eq-ch3-decay-assumption-general1}
|u(x',t)-u_0(x'_0)|\le \big(A_0+[u_0]_{C^\alpha(B'_1)}\big)(|x'-x'_0|^2+t^2)^{\frac\alpha2}.
\end{equation}

\begin{proof} 
Set  
$$F=|u|_{L^\infty(G_1)}+|f|_{C^\alpha(\bar G_1)}.$$
We take any $x_0=(x_0', t_0)\in G_{1/2}$. 
As in the proof of Lemma \ref{lemma-ch3-Linear-Estimate-alpha}, 
we have 
\begin{equation*}
L(u-u_0(x'_0))= h,\end{equation*}
where
$$h=f-f(x'_0,0)-u_0(x'_0)(c-c(x'_0,0)).$$
Hence, for any $(x',t)\in G_1$, 
$$|h(x',t)|\le CF(|x'-x_0'|^2+t^2)^{\frac\alpha2}.$$
Consider \eqref{eq-ch3-Equ-alpha-general}
in $B_{3t_0/4}(x_0)$. Note that $t_0/4\le t\le 7t_0/4$, for any $(x', t)\in B_{3t_0/4}(x_0)$. 
Hence, 
\begin{align*}\label{eq-ScalingAssumption}
|t^2a_{ij}|^*_{C^\alpha(B_{3t_0/4}(x_0))}
+t_0|tb_i|^*_{C^\alpha(B_{3t_0/4}(x_0))}
+t_0^{2}|c|^*_{C^\alpha(B_{3t_0/4}(x_0))}
\le Ct_0^2.
\end{align*}
Moreover, for any $\xi\in\mathbb R^n$ and any $x\in B_{3t_0/4}(x_0)$, 
\begin{align*}t^2a_{ij}(x)\xi_i\xi_j
\ge \frac{1}{16}\lambda t_0^2|\xi|^2.
\end{align*}
We write \eqref{eq-ch3-Equ-alpha-general} as 
$$t_0^{-2}L(u-u_0(x_0'))=t_0^{-2}h.$$
Note that
$$|h|^*_{C^\alpha(B_{3t_0/4}(x_0))}=|h|_{L^\infty(B_{3t_0/4}(x_0))}
+t_0^{\alpha}[h]_{C^\alpha(B_{3t_0/4}(x_0))}\le CFt_0^{\alpha},$$
and, by \eqref{eq-ch3-decay-assumption-general1}, 
$$|u-u_0(x'_0)|_{L^\infty(B_{3t_0/4}(x_0))}\le C(A_0+F)t_0^\alpha.$$
The scaled interior $C^{2,\alpha}$-estimate implies 
\begin{align*} 
&t_0^\alpha[u]_{C^\alpha(B_{t_0/2}(x_0))}
+t_0|Du|^*_{C^\alpha(B_{t_0/2}(x_0))}
+t_0^{2}|D^2u|^*_{C^\alpha(B_{t_0/2}(x_0))}\\
&\quad \le C\big\{|u-u_0(x'_0)|_{L^\infty(B_{3t_0/4}(x_0))}
+t_0^{2}\, t_0^{-2}|h|^*_{C^\alpha(B_{3t_0/4}(x_0))}\big\}
\le C(A_0+F)t_0^\alpha.
\end{align*}
This holds for any $x_0=(x_0', t_0)\in G_{1/2}$. 
Then, by \eqref{eq-ch3-decay-assumption-general} and evaluating at $x_0'=x'$
the above estimate of the 
$L^\infty$-norms of $D u$ and $D^2u$,  
we get, 
for any $(x', t)\in G_{1/2}$, 
\begin{equation}\label{eq-ch3-linear-estimate-alpha-general}
|u(x',t)-u_0(x')|+t|Du(x',t)|+t^2|D^2u(x',t)|\le C(A_0+F) t^\alpha.\end{equation}
The estimate of the H\"older semi-norms implies, for any $(x', t)\in G_{1/2}$,  
\begin{equation*}
[u]_{C^\alpha(B_{t/2}(x',t))}
+t[Du]_{C^\alpha(B_{t/2}(x',t))}+t^{2}[D^2u]_{C^\alpha(B_{t/2}(x',t))}
\le C(A_0+F),\end{equation*}
and hence 
\begin{equation*}
[u]_{C^\alpha(B_{t/2}(x',t))}
+[sDu]_{C^\alpha(B_{t/2}(x',t))}+[s^2D^2u]_{C^\alpha(B_{t/2}(x',t))}
\le C(A_0+F).\end{equation*}
Therefore, we obtain the desired result by applying Lemma \ref{lemma-ch3-GlobalHolder}
to $u$, $tDu$, and $t^2D^2u$. 
\end{proof}

We are ready to prove a result concerning the H\"older continuity up to $\Sigma_1$.

\begin{theorem}\label{thrm-ch3-Linear-Regularity-0-alpha}
For some constant 
$\alpha\in (0,1)$, assume $a_{ij}, b_i, c\in C^{\alpha}(\bar G_1)$, 
with \eqref{eq-ch3-ellipticity-t}, and $c\le -c_0$ 
and $Q({\alpha})\le -c_{\alpha}$  in $G_1$, for some positive constants 
$c_0$ and $c_{\alpha}$.
For some $f\in C^\alpha(\bar G_1)$, let $u\in L^\infty(G_1)\cap C^2(G_1)$ be 
a solution of \eqref{eq-ch3-Equ}. 
Then, for any $r\in (0,1)$, 
\begin{equation}\label{eq-ch3-regularity-alpha}
u, tDu, t^2D^2u\in C^\alpha(\bar G_r),\end{equation}
with $u=u_0$, $tDu=0$, and $t^2D^2u=0$ on $\Sigma_1$, and 
\begin{align}\label{eq-ch3-estimate-alpha-alpha}
|u|_{C^\alpha(\bar G_{1/2})}+|tDu|_{C^\alpha(\bar G_{1/2})}
+|t^2D^2u|_{C^\alpha(\bar G_{1/2})}
\le C\big\{|u|_{L^\infty(G_1)}+|f|_{C^\alpha(\bar G_1)}\big\},\end{align}
where $C$ is a positive constant depending only on $n$, $\lambda$, 
$\alpha$,  $c_0$, $c_\alpha$, and
the $C^\alpha$-norms of $a_{ij}, b_i, c$ in $\bar G_1$.
\end{theorem}

\begin{proof} 
By Lemma \ref{lemma-ch3-Linear-Estimate-alpha}, 
\eqref{eq-ch3-linear-estimate-alpha} holds and implies the assumption 
\eqref{eq-ch3-decay-assumption-general} with 
$$A_0=C\{|u|_{L^\infty(G_1)}+|f|_{C^\alpha(\bar G_1)}\}.$$ 
Hence, Lemma \ref{lemma-ch3-Linear-Regularity-general} 
implies the desired result. \end{proof} 

According to \eqref{eq-ch3-linear-estimate-alpha-general} 
in the proof of Lemma \ref{lemma-ch3-Linear-Regularity-general}, we have, 
for any $(x', t)\in G_{1/2}$,  
\begin{align}\label{eq-ch3-linear-decay-estimate-alpha-general}
|u(x',t)-u_0(x')|+t|Du(x',t)|+t^2|D^2u(x',t)|
\le Ct^\alpha\{|u|_{L^\infty(G_1)}+|f|_{C^\alpha(\bar G_1)}\big\}.
\end{align}

\section{The Tangential Regularity}\label{sec-OptimalRegularity-linear-tangential-regularity} 

Now, we start to study the regularity of solutions 
along tangential directions. We will prove that 
solutions are as regular tangentially as allowed near the boundary.

We first describe our approach. 
We aim to prove $D_{x'}u\in C^\alpha(\bar G_r)$, for any $r\in (0,1)$, under the assumption 
$a_{ij}, b_i, c, f, D_{x'}a_{ij}, D_{x'}b_i, D_{x'}c, D_{x'}f\in C^{\alpha}(\bar G_1)$.
The interior Schauder theory (along the tangential directions) implies $D_{x'}D^2u\in C^\alpha(G_1)$. 
Fix a unit vector $e\in \mathbb R^{n-1}\times\{0\}$. 
A simple differentiation of  $Lu=f$ along the direction $e$ yields 
$$L(\partial_eu)+t^2\partial_ea_{ij}\partial_{ij}u
+t\partial_eb_{i}\partial_{i}u+\partial_ecu=\partial_ef.$$
By 
Theorem \ref{thrm-ch3-Linear-Regularity-0-alpha}, 
or \eqref{eq-ch3-regularity-alpha} in particular, we 
have  $t^2\partial_ea_{ij}\partial_{ij}u, 
t\partial_eb_{i}\partial_{i}u$, $\partial_ecu\in 
C^\alpha(\bar G_r)$, for any $r\in (0,1)$. 
We now move these terms to the right-hand side and write 
\begin{equation}\label{eq-ch3-equation-u-k}L(\partial_eu)=f_1,\end{equation}
where 
\begin{equation}\label{eq-ch3-equation-u-k-RHS}f_1=\partial_ef-t^2\partial_ea_{ij}\partial_{ij}u
-t\partial_eb_{i}\partial_{i}u-\partial_ecu.\end{equation}
Then,  $f_1\in 
C^\alpha(\bar G_r)$, for any $r\in (0,1)$. By Theorem \ref{thrm-ch3-Linear-Regularity-0-alpha} again, 
$tDu=0$ and $t^2D^2u=0$ on $\Sigma_1$. Hence,  
$$\frac{f_1}{c}=\frac{1}{c}(\partial_ef-\partial_ecu)=\frac{1}{c}\Big(\partial_ef-\frac{f}{c}\partial_ec\Big)
=\partial_eu_0\quad\text{on }\Sigma_1.$$
If we can prove that $\partial_eu$ is bounded in $G_1$, 
then we can apply 
Theorem \ref{thrm-ch3-Linear-Regularity-0-alpha} 
to \eqref{eq-ch3-equation-u-k} to conclude
$\partial_eu\in C^\alpha(\bar G_r)$, 
for any $r\in (0,1)$.
This will be the desired tangential regularity. 

We now prove the boundedness of tangential derivatives.

\begin{lemma}\label{lemma-ch3-Linear-C{1}Estimate}
For some  constant 
$\alpha\in (0,1)$, assume $a_{ij}, b_i, c\in C^{\alpha}(\bar G_1)$, 
with $D_{x'}a_{ij}$, $D_{x'}b_i$, $D_{x'}c\in C(\bar G_1)$, 
\eqref{eq-ch3-ellipticity-t}, and $c\le -c_0$ 
and $Q(\alpha)\le -c_\alpha$  in $G_1$, for some positive constants 
$c_0$ and $c_\alpha$.
For some $f\in C^{\alpha}(\bar G_1)$ with $D_{x'}f\in C(\bar G_1)$, 
let $u\in C(\bar G_1)\cap C^2(G_1)$ be 
a solution of \eqref{eq-ch3-Equ} and \eqref{eq-ch3-Dirichlet}.  
Then, $D_{x'}u$ is bounded in $G_r$, for any $r\in (0,1)$, and 
\begin{align}\label{eq-ch3-estimate-tangential-boundedness}
|D_{x'}u|_{L^\infty(G_{1/2})}\le C\big\{|u|_{L^\infty(G_1)}
+|f|_{C^{\alpha}(\bar G_1)}+|D_{x'}f|_{L^\infty(G_1)}\big\},\end{align}
where $C$ is a positive constant depending only on $n$,  $\lambda$, 
$\alpha$, $c_0$, $c_\alpha$,
the $C^{\alpha}$-norms of $a_{ij}, b_i, c$ in $G_1$, 
and the $L^\infty$-norms of $D_{x'}a_{ij}, D_{x'}b_i, D_{x'}c$ in $G_1$.
\end{lemma}

\begin{proof} We fix a unit vector $e\in \mathbb R^{n-1}\times\{0\}$
and prove that $\partial_eu$ is bounded in $G_r$, for any $r\in (0,1)$. Set 
\begin{align*}F=|u|_{L^\infty(G_1)}+|f|_{C^\alpha(\bar G_1)}+|D_{x'}f|_{L^\infty(G_1)}.\end{align*}
Take any $\tau$ small and define 
$$u_\tau(x)=\frac{1}{\tau}[u(x+\tau e)-u(x)],$$ 
and similarly $a_{ij,\tau}$, $b_{i,\tau}$, $c_\tau$, and $f_\tau$. Evaluate the equation $Lu=f$
at $x+\tau e$ and $x$, take the difference, and then divide by $\tau$. Hence, 
\begin{equation}\label{eq-ch3-equation-1alpha1-bounded}
Lu_\tau=h_1,\end{equation}
where 
\begin{equation}\label{eq-ch3-equation-1alpha2-bounded}
h_1=f_\tau-t^2a_{ij,\tau}\partial_{ij}u(\cdot+\tau e)
-tb_{i,\tau}\partial_{i}u(\cdot+\tau e)-c_\tau u(\cdot+\tau e).\end{equation}
By  \eqref{eq-ch3-linear-decay-estimate-alpha-general}, 
we have, for any $x=(x',t)\in G_{7/8}$,  
\begin{equation}\label{eq-ch3-equation-1alpha3-bounded}
t|Du(x)|+t^2|D^2u(x)|
\le CFt^\alpha.\end{equation}
This implies 
$$|h_1|_{L^\infty(G_{3/4})}\le CF,$$
where $C$ is a positive constant 
independent of $\tau$. 

Take a cutoff function $\eta=\eta(x')\in C_0^\infty(B'_{3/4})$, with $\eta=1$ in $B'_{1/2}$. Then, 
\begin{equation}\label{eq-ch3-equation-1alpha4-bounded}
L(\eta u_\tau)=\eta h_1+h_2,\end{equation}
where 
$$h_2=2t^2a_{ij}\partial_{i}\eta \partial_ju_\tau
+(ta_{ij}\partial_{ij}\eta+b_i\partial_i\eta)tu_\tau.$$
Similarly, we have 
$$|h_2|_{L^\infty(G_{3/4})}\le CF.$$
We now examine $\eta u_\tau$ on $\partial G_{3/4}$. 
First, $\eta u_\tau=0$ on $\partial B'_{3/4}\times (0,3/4)$. Next, 
the interior $C^1$-estimate (applied to the equation $Lu=f$) implies 
$$|Du|_{L^\infty(B'_{7/8}\times \{3/4\})}\le C\big\{|u|_{L^\infty(G_1)}+|f|_{L^\infty(G_1)}\big\}\le CF.$$
Hence, 
$$|\eta u_\tau|_{L^\infty(B'_{3/4}\times \{3/4\})}\le  CF.$$
Last, on $\Sigma_{3/4}$, we have $u=u_0$ and hence 
$$|\eta u_\tau|_{L^\infty(\Sigma_{3/4})}\le C\big\{|f|_{L^\infty(G_1)}+|D_{x'}f|_{L^\infty(G_1)}\big\}\le CF.$$
Note that $c\le -c_0<0$ in \eqref{eq-ch3-equation-1alpha4-bounded}. By applying the maximum principle and 
considering the maximum and minimum of $\eta u_\tau$ 
in $G_{3/4}$ and on $\partial G_{3/4}$ separately, we obtain 
$$|\eta u_\tau|_{L^\infty(G_{3/4})}\le |\eta u_\tau|_{L^\infty(\partial G_{3/4})}
+C\big\{|h_1|_{L^\infty(G_{3/4})}+|h_2|_{L^\infty(G_{3/4})}\big\}\le CF,$$
where $C$ is a positive constant independent of $\tau$. By letting $\tau\to 0$, we get 
\begin{equation}\label{eq-equation-1alpha5-bounded}
|\partial_eu|_{L^\infty(G_{1/2})}\le CF.\end{equation}
This implies \eqref{eq-ch3-estimate-tangential-boundedness}. 
\end{proof}

We are ready to prove the regularity of tangential derivatives. 
In the following, $\tau$ is always a nonnegative integer.

\begin{theorem}\label{thrm-ch3-Linear-TangentialEstimate-general}
For some integer $\ell\ge 1$ and constant 
$\alpha\in (0,1)$, 
assume $D^\tau_{x'}a_{ij}$, $D^\tau_{x'}b_i$, $D^\tau_{x'}c\in C^{\alpha}(\bar G_1)$
for any $\tau\le \ell$, 
with \eqref{eq-ch3-ellipticity-t}, 
and $c\le -c_0$ 
and $Q(\alpha)\le -c_\alpha$  in $G_1$, for some positive constants 
$c_0$ and $c_\alpha$. 
For some $f$ with 
$D^\tau_{x'}f\in C^{\alpha}(\bar G_1)$ for any $\tau\le\ell$, let $u\in C(\bar G_1)\cap C^2(G_1)$ be 
a solution of \eqref{eq-ch3-Equ} and \eqref{eq-ch3-Dirichlet}.  
Then, for any  $\tau\le \ell$, and any $r\in (0,1)$, 
\begin{equation}\label{eq-ch3-regularity-tangential-linear}
D^\tau_{x'}u, tDD^\tau_{x'}u, t^2D^2D^\tau_{x'}u\in C^\alpha(\bar G_r),\end{equation}
with $D^\tau_{x'}u=D_{x'}^\tau u_0$, $tDD^\tau_{x'}u=0$, and $t^2D^2D^\tau_{x'}u=0$ on $\Sigma_1$, and 
\begin{align}\label{eq-ch3-estimate-tangential-linear}\begin{split}
&|D^\tau_{x'}u|_{C^\alpha(\bar G_{1/2})}+|tDD^\tau_{x'}u|_{C^\alpha(\bar G_{1/2})}
+|t^2D^2D^\tau_{x'}u|_{C^\alpha(\bar G_{1/2})}\\
&\qquad
\le C\Big\{|u|_{L^\infty(G_1)}+\sum_{i=0}^\ell |D_{x'}^if|_{C^{\alpha}(\bar G_1)}\Big\},
\end{split}\end{align}
where $C$ is a positive constant depending only on $n$, $\ell$, $\lambda$, 
$\alpha$, $c_0$, $c_\alpha$, and
the $C^{\alpha}$-norms of $D_{x'}^\tau a_{ij}, D_{x'}^\tau b_i, D_{x'}^\tau c$ in $\bar G_1$ for $\tau\le\ell$.
\end{theorem}

\begin{proof} The proof is based on an induction on $\ell$. We first consider $\ell=1$. 
Set 
$$F=|u|_{L^\infty(G_1)}+|f|_{C^{\alpha}(\bar G_1)}
+|D_{x'}f|_{C^{\alpha}(\bar G_1)}.$$ 
Fix a unit vector $e\in \mathbb R^{n-1}\times\{0\}$. 
Then, $\partial_eu$ is a solution of \eqref{eq-ch3-equation-u-k}, 
with $f_1$ given by \eqref{eq-ch3-equation-u-k-RHS}. 
First, by Lemma \ref{lemma-ch3-Linear-C{1}Estimate}, 
we get $\partial_eu\in L^\infty(G_r)$ for any $r\in (0,1)$. 
Next, by  $a_{ij}, b_i, c, f, D_{x'}a_{ij}, D_{x'}b_i, D_{x'}c, D_{x'}f\in C^{\alpha}(\bar G_1)$ and 
Theorem \ref{thrm-ch3-Linear-Regularity-0-alpha}, 
or \eqref{eq-ch3-regularity-alpha} in particular, 
we have $f_1\in 
C^\alpha(\bar G_r)$, for any $r\in (0,1)$, and 
$$|f_1|_{C^{\alpha}(\bar G_r)}\le CF.$$
By \eqref{eq-ch3-equation-u-k-RHS} and $tDu=0$ and $t^2D^2u=0$ on $\Sigma_1$, we have 
$$\partial_eu_0=\frac{f_1}{c}(\cdot, 0)\quad\text{on }B'_1.$$
By applying Theorem \ref{thrm-ch3-Linear-Regularity-0-alpha} to the equation \eqref{eq-ch3-equation-u-k}, 
we conclude, for any $r\in (0,1)$, 
\begin{equation*}
\partial_eu, tD\partial_eu, t^2D^2\partial_eu\in C^\alpha(\bar G_r),\end{equation*}
with $\partial_eu=\partial_eu_0$, 
$tD\partial_eu=0$ and $t^2D^2\partial_eu=0$ on $\Sigma_1$, and 
\begin{align*}
|\partial_eu|_{C^\alpha(\bar G_{1/2})}+|tD\partial_eu|_{C^\alpha(\bar G_{1/2})}
+|t^2D^2\partial_eu|_{C^\alpha(\bar G_{1/2})}\le CF.
\end{align*}
This implies the desired result for $\ell=1$. 

The proof for the general $\ell$ is based on induction and hence omitted. 
\end{proof}

\section{Regularity along the Normal Direction}\label{sec-OptimalRegularity-normal-regularity}

In this section, we discuss regularity along the normal direction. 
We will prove that 
solutions are regular along the normal direction up to a certain order.

We first describe our approach. 
We aim to prove $\partial_tu\in C^\alpha(\bar G_r)$ for any $r\in (0,1)$, under the assumption 
$a_{ij}, b_i, c, f\in C^{1,\alpha}(\bar G_1)$.
The interior Schauder theory implies $u\in C^{3,\alpha}(G_1)$. 
A simple differentiation of  $Lu=f$ with respect to $t$ yields 
\begin{align*}
L(\partial_tu)+t^2\partial_{t}a_{ij}\partial_{ij}u+t\partial_tb_i\partial_iu+\partial_tcu+2ta_{ij}\partial_{ij}u+b_i\partial_iu=
\partial_tf.\end{align*}
By $f\in C^{1,\alpha}(\bar G_1)$, we have 
$\partial_tf\in C^{\alpha}(\bar G_1)$.
Next, by 
$a_{ij}, b_i, c\in C^{1,\alpha}(\bar G_1)$, Theorem \ref{thrm-ch3-Linear-Regularity-0-alpha}
and Theorem \ref{thrm-ch3-Linear-TangentialEstimate-general},  
or \eqref{eq-ch3-regularity-alpha} and \eqref{eq-ch3-regularity-tangential-linear} with $\ell=\tau=1$
in particular, we 
have 
$t^2\partial_{t}a_{ij}\partial_{ij}u$, $t\partial_tb_i\partial_iu$, $\partial_tcu\in 
C^\alpha(\bar G_r)$ and  
$ta_{i\beta}\partial_{i\beta}u$, $b_\beta\partial_\beta u\in 
C^\alpha(\bar G_r)$, for any $r\in (0,1)$. 
We now move these terms to the right-hand side and combine the remaining two terms 
$2ta_{in}\partial_{it}u$ and 
$b_n\partial_tu$ with $L(\partial_tu)$. 
(In fact, we write $4a_{n\beta}\partial_{t\beta}u=2a_{n\beta}\partial_{t\beta}u+2a_{n\beta}\partial_{t\beta}u$, 
keep one $2a_{n\beta}\partial_{t\beta}u$ in the left-hand side, and move the other $2a_{n\beta}\partial_{t\beta}u$
to the right-hand side.)
Hence, 
\begin{align*}
&t^2a_{ij}\partial_{ij}(\partial_tu)+t(b_i+2a_{in})\partial_i(\partial_tu)+(c+b_n)(\partial_tu)\\
&\qquad=
\partial_tf-t^2\partial_{t}a_{ij}\partial_{ij}u-t\partial_tb_i\partial_iu-\partial_tcu
-2ta_{i\beta}\partial_{i\beta}u-b_\beta \partial_\beta u.\end{align*}
We regard the left-hand side as an operator acting on $\partial_tu$. Set 
\begin{align*} 
L^{(1)}=t^2a_{ij}\partial_{ij}+t(b_i+2a_{in})\partial_i+(c+b_n).\end{align*}
Then, 
\begin{equation}\label{eq-ch3-equation-u-t}L^{(1)}(\partial_tu)=f_1,\end{equation}
where 
\begin{equation}\label{eq-ch3-equation-u-t-RHS}f_1=
\partial_tf-t^2\partial_{t}a_{ij}\partial_{ij}u-t\partial_tb_i\partial_iu-\partial_tcu
-2ta_{i\beta}\partial_{i\beta}u-b_\beta \partial_\beta u.\end{equation}
Hence, 
$f_1\in 
C^\alpha(\bar G_r)$, for any $r\in (0,1)$. 
We note that $L^{(1)}$ has a similar structure as $L$ and thus has an associated 
$Q^{(1)}(\mu)$ given by 
$$Q^{(1)}(\mu)=\mu(\mu-1)a_{nn}+\mu(b_n+2a_{nn})+(c+b_n).$$
Then, 
\begin{equation}\label{eq-ch3-relation-P-1}
Q^{(1)}(\mu)=Q({\mu+1}).\end{equation}
Next, we have 
$tDu=0$ and $t^2D^2u=0$ on $\Sigma_1$ by Theorem \ref{thrm-ch3-Linear-Regularity-0-alpha} 
and $tDD_{x'}u=0$ on $\Sigma_1$ by 
Theorem \ref{thrm-ch3-Linear-TangentialEstimate-general} with $\ell=\tau=1$. 
Then, 
$$\frac{f_1}{b_n+c}(\cdot, 0)=u_1\quad\text{on }B'_1,$$
where $u_1$ is defined by
\begin{equation}\label{eq-ch3-derivative-u-t-value}
u_1=\frac{\partial_tf-\partial_tcu-b_\beta\partial_\beta u}{b_n+c}(\cdot, 0)\quad\text{on }B'_1.
\end{equation}
If we can prove that $\partial_tu$ is continuous up to $\Sigma_1$ with $\partial_tu=u_1$ on $\Sigma_1$, 
then we can apply Theorem \ref{thrm-ch3-Linear-Regularity-0-alpha}
to \eqref{eq-ch3-equation-u-t} to conclude 
$\partial_tu\in C^\alpha(\bar G_r)$, 
for any $r\in (0,1)$, 
under an additional assumption $Q^{(1)}(\alpha)<0$ in $\bar G_1$.
This will be the desired regularity along the normal direction. 
We now carry out the above outline. 

We first prove a decay estimate of the normal derivative. 

\begin{lemma}\label{lemma-ch3-Linear-C{1,alpha}Estimate-normal}
For some  constant 
$\alpha\in (0,1)$, assume $a_{ij}, b_i, c\in C^{1,\alpha}(\bar G_1)$,  
with \eqref{eq-ch3-ellipticity-t}, and $c\le -c_0$ 
and $Q({1+\alpha})\le -c_{1+\alpha}$  in $G_1$, for some positive constants 
$c_0$ and $c_{1+\alpha}$. 
For some $f\in C^{1,\alpha}(\bar G_1)$, let $u\in C(\bar G_1)\cap C^2(G_1)$ be 
a solution of \eqref{eq-ch3-Equ} and \eqref{eq-ch3-Dirichlet}.  
Then, for any $(x',t)\in B_{1/2}'\times (0,1/2)$ and any $x_0'\in B'_{1/2}$, 
\begin{align}\label{eq-ch3-estimate-decau-2-k}\begin{split}
|\partial_tu(x',t)-u_1(x'_0)|
\le C\big\{|u|_{L^\infty(G_1)}+|f|_{C^{1,\alpha}(\bar G_1)}\big\}(|x'-x_0'|^2+t^2)^{\frac{\alpha}{2}},
\end{split}\end{align}
where $u_1$ is given by \eqref{eq-ch3-derivative-u-t-value} and 
$C$ is a positive constant depending only on $n$,  $\lambda$, 
$\alpha$, $c_0$, $c_{1+\alpha}$, and
the $C^{1,\alpha}$-norms of $a_{ij}, b_i, c$ in $\bar G_1$.
In particular, $\partial_tu\in C(\bar G_r)$, for any $r\in (0,1)$. 
\end{lemma}

By $Q(0)=c\le -c_0$ 
and $Q({1+\alpha})\le -c_{1+\alpha}$  in $G_1$, we have 
$Q(1)=b_n+c<0$ in $\bar G_1$. 
Hence, the expression in the right-hand side of \eqref{eq-ch3-derivative-u-t-value} makes sense. 

\begin{proof} The proof consists of two steps. Set 
$$F=|u|_{L^\infty(G_1)}+|f|_{C^{1,\alpha}(\bar G_1)}.$$

{\it Step 1.} We first prove a decay estimate of $u$. 
For any fixed $x_0'\in B_{1/2}'$, consider a linear function 
$$l_{x_0'}(x',t)=l_0+l_\beta( x_\beta-x_{0\beta})+l_nt,$$
for some constants $l_0, l_1, \cdots, l_{n-1}$, and $l_n$, to be determined. 
Then, 
$$L(u-l_{x_0'})=f-t(b_\beta l_\beta+b_nl_n)-cl_0-cl_\beta ( x_\beta-x_{0\beta})-cl_nt.$$
The linear part of the expression in the right-hand side at $(x_0', 0)$ is given by 
\begin{align*}
&f(x_0',0)-c(x_0',0)l_0+\big(\partial_\beta f(x_0',0)-\partial_\beta c(x_0',0)l_0-c(x_0',0)l_\beta\big)( x_\beta-x_{0\beta})\\
&\qquad +\big(\partial_tf(x_0',0)-\partial_tc(x_0',0)l_0-c(x_0',0)l_n-b_\beta(x_0',0)l_\beta-b_n(x_0',0)l_n\big)t.\end{align*}
We now make this equal to zero by choosing $l_0$, $l_\beta$, and $l_n$ successively. 
A simple computation yields 
\begin{align*} 
l_0&=\frac{f(x_0',0)}{c(x_0',0)}=u_0(x_0'),\\
l_\beta&=\frac{\partial_\beta f(x_0',0)-\partial_\beta c(x_0',0)l_0}{c(x_0',0)}=\partial_\beta u_0(x_0'),\end{align*} 
and 
$$l_n=\frac{\partial_tf(x_0',0)-\partial_tc(x_0',0)l_0-b_\beta(x_0',0)l_\beta}{b_n(x_0',0)+c(x_0',0)}=u_1(x_0'),$$
where $u_1$ is given by \eqref{eq-ch3-derivative-u-t-value}. Hence, 
$$l_{x_0'}(x',t)=u_0(x_0')+\partial_\beta u_0(x_0')( x_\beta-x_{0\beta})+u_1(x_0')t.$$
With such a choice of $l_{x_0'}$, we have 
\begin{equation}\label{eq-ch3-Linear-Estimate-1-0}
L(u-l_{x_0'})=h,\end{equation}
such that, for any $(x',t)\in  B'_{1/2}(x_0')\times (0,1/2)$, 
\begin{equation}\label{eq-ch3-Linear-Estimate-1-1}
|h(x',t)| \le CF(|x'-x_0'|^2+t^2)^{\frac{1+\alpha}{2}}.\end{equation} 
Moreover, we have, for any $x'\in  B_{1/2}'(x_0')$, 
\begin{equation}\label{eq-ch3-Linear-Estimate-1-2}
|u(x',0)-l_{x_0'}(x,0)|=|u_0(x')-u_0(x_0')-\partial_\beta u_0(x_0')( x_\beta-x_{0\beta})|\le CF|x'-x_0'|^{1+\alpha}.\end{equation}
Next, take $\psi$ as in Lemma \ref{lemma-ch3-Linear-ComparisonFunction-t-factor} with
$\sigma=0$ and $\mu=1+\alpha$. 
By applying the maximum principle as in the proof of 
Lemma \ref{lemma-ch3-Linear-Estimate-alpha}, we have, 
for any $(x,t)\in B_{1/2}'(x_0')\times (0,1/2)$,
\begin{align}\label{eq-ch3-estimate-decau-1-k}
\big|u(x',t)-l_{x_0'}(x',t)\big|
\le CF
(|x'-x_0'|^2+t^2)^{\frac{1+\alpha}{2}}.\end{align}
With $x'=x_0'$, we obtain, for any $(x',t)\in G_{1/2}$, 
$$|u(x',t)-u_0(x')-u_1(x')t|\le CFt^{1+\alpha}.$$
Dividing by $t$ and letting $t\to0$, we conclude that $\partial_tu(x',0)$ exists and is equal to $u_1(x')$, 
for any $x'\in B_{1/2}'$. 

{\it Step 2.} We prove \eqref{eq-ch3-estimate-decau-2-k}. 
We take any $x_0=(x_0', t_0)\in G_{1/2}$ and consider \eqref{eq-ch3-Linear-Estimate-1-0}
in $B_{3t_0/4}(x_0)$. Note that $t_0/4\le t\le 7t_0/4$, for any $(x', t)\in B_{3t_0/4}(x_0)$. 
Hence, 
\begin{align*}
|t^2a_{ij}|^*_{C^\alpha(B_{3t_0/4}(x_0))}
+t_0|tb_i|_{L^\infty(B_{3t_0/4}(x_0))}
+t_0^2|c|_{L^\infty(B_{3t_0/4}(x_0))}
\le Ct_0^2,\end{align*}
where $|\cdot|^*_{C^\alpha(B_{3t_0/4}(x_0))}$ is the scaled H\"older norm 
defined by \eqref{eq-ch2-scaled-Holder-norm}. 
Moreover, by \eqref{eq-ch3-Linear-Estimate-1-1}, 
$$|h|_{L^\infty(B_{3t_0/4}(x_0))}
\le CFt_0^{1+\alpha},$$
and, by \eqref{eq-ch3-estimate-decau-1-k}, 
$$|u-l_{x'_0}|_{L^\infty(B_{3t_0/4}(x_0))}\le CFt_0^{1+\alpha}.$$
We write \eqref{eq-ch3-Linear-Estimate-1-0} as
$$t_0^{-2}L(u-l_{x_0'})=t_0^{-2}h.$$
The scaled interior $C^{1}$-estimate implies 
\begin{align*}
&t_0|D(u-l_{x'_0})|_{L^\infty(B_{t_0/2}(x_0))}
\\&\qquad
\le C\big\{|u-l_{x'_0}|_{L^\infty(B_{3t_0/4}(x_0))}
+t_0^2\, t_0^{-2}|h|_{L^\infty(B_{3t_0/4}(x_0))}\big\}
\le CFt_0^{1+\alpha},
\end{align*}
and hence
$$
|D(u-l_{x'_0})|_{L^\infty(B_{t_0/2}(x_0))}
\le CFt_0^{\alpha}.
$$
By considering $\partial_t$ only and evaluating at $x_0$, we get, for any $(x',t)\in G_{1/2}$, 
$$|\partial_tu(x',t)-u_1(x')|\le CFt^\alpha.$$
By $u_1\in C^\alpha(B_1')$, we obtain \eqref{eq-ch3-estimate-decau-2-k} easily. 
Hence, $\partial_tu$ is continuous up to $\Sigma_1$. 
\end{proof}

We now make two remarks. 
First, \eqref{eq-ch3-derivative-u-t-value} and \eqref{eq-ch3-estimate-decau-2-k}
show that $\partial_tu(\cdot, 0)$ is determined by the equation. 
Second, for a fixed $x_0'\in B_{1/2}'$, the estimate \eqref{eq-ch3-estimate-decau-1-k}
actually establishes that $u$ is $C^{1,\alpha}$ at $(x_0',0)$. In fact, it holds under a weaker assumption 
that $f$ and $c$ are $C^{1,\alpha}$ at $(x_0',0)$ and $b_i$ is $C^{\alpha}$ at $(x_0',0)$. 
More generally, we can prove that $u$ is $C^{m,\alpha}$ at $(x_0',0)$ if $Q(m+\alpha)<0$ in $\bar G_1$ 
and under appropriate assumptions on the coefficients and $f$. 

We are ready to prove the regularity of normal derivatives.

\begin{theorem}\label{thrm-ch3-Linear-NormalEstimate-general}
For some integers $\ell\ge m\ge 1$ and constant 
$\alpha\in (0,1)$, assume $a_{ij}, b_i, c\in C^{\ell,\alpha}(\bar G_1)$,  
with \eqref{eq-ch3-ellipticity-t}, and $c\le -c_0$ 
and $Q({m+\alpha})\le -c_{m+\alpha}$  in $G_1$, for some positive constants 
$c_0$ and $c_{m+\alpha}$.  
For some $f\in C^{\ell,\alpha}(\bar G_1)$, let $u\in C(\bar G_1)\cap C^2(G_1)$ be 
a solution of \eqref{eq-ch3-Equ} and \eqref{eq-ch3-Dirichlet}.  
Then, for any nonnegative integers $\nu$ and $\tau$ with $\nu\le m$ and 
$\nu+\tau\le \ell$, and any $r\in (0,1)$, 
\begin{equation*}
\partial_t^\nu D^\tau_{x'}u, tD\partial_t^\nu D^\tau_{x'}u, t^2D^2\partial_t^\nu D^\tau_{x'}u
\in C^\alpha(\bar G_r),\end{equation*}
with $tD\partial_t^\nu D^\tau_{x'}u=0$ and $t^2D^2\partial_t^\nu D^\tau_{x'}u=0$ on $\Sigma_1$,
and 
\begin{align*}
&|\partial_t^\nu D^\tau_{x'}u|_{C^\alpha(\bar G_{1/2})}+|tD\partial_t^\nu D^\tau_{x'}u|_{C^\alpha(\bar G_{1/2})}
+|t^2D^2\partial_t^\nu D^\tau_{x'}u|_{C^\alpha(\bar G_{1/2})}\\
&\qquad\le C\big\{|u|_{L^\infty(G_1)}+|f|_{C^{\ell,\alpha}(\bar G_1)}\big\},
\end{align*}
where 
$C$ is a positive constant depending only on $n$, $\ell$, $\lambda$, 
$\alpha$, $c_0$, $c_{m+\alpha}$, and
the $C^{\ell,\alpha}$-norms of $a_{ij}, b_i, c$ in $\bar G_1$. In particular, 
$u\in C^{m,\alpha}(\bar G_r)$, for any $r\in (0,1)$. 
\end{theorem}

\begin{proof} The proof is based on an induction on $m$. We first consider $m=1$. 
We note that $\partial_tu$ is continuous up to $\Sigma_1$, by 
Lemma \ref{lemma-ch3-Linear-C{1,alpha}Estimate-normal},
and satisfies \eqref{eq-ch3-equation-u-t}, where $f_1$ is given by \eqref{eq-ch3-equation-u-t-RHS}. 
By Theorem \ref{thrm-ch3-Linear-TangentialEstimate-general}, 
\eqref{eq-ch3-regularity-tangential-linear} and 
\eqref{eq-ch3-estimate-tangential-linear} hold for any $\tau\le\ell$. 
By 
$a_{ij}, b_i, c, f\in C^{\ell,\alpha}(\bar G_1)$, we 
have 
$D_{x'}^\tau f_1\in C^\alpha(\bar G_r)$, for any $\tau\le \ell-1$ and any $r\in (0,1)$, 
and 
$$\sum_{\tau=0}^{\ell-1}
|D_{x'}^\tau f_1|_{C^{\alpha}(\bar G_r)}\le C\big\{|u|_{L^\infty(G_1)}+|f|_{C^{\ell,\alpha}(\bar G_1)}\big\}.$$
By \eqref{eq-ch3-equation-u-t-RHS}, \eqref{eq-ch3-derivative-u-t-value}, and $tDu=0$, 
$t^2D^2u=0$, and $tDD_{x'}u=0$ on $\Sigma_1$, we get
\begin{equation}\label{eq-ch3-normal-equation-condition}
\partial_tu(\cdot, 0)=u_1=\frac{f_1}{b_n+c}(\cdot, 0)\quad\text{on }B'_1.\end{equation}
By \eqref{eq-ch3-relation-P-1} and 
$Q({1+\alpha})\le -c_{1+\alpha}$ in $G_1$, 
we have $Q^{(1)}(\alpha)\le -c_{1+\alpha}$ in $G_1$. 
By applying 
Theorem \ref{thrm-ch3-Linear-TangentialEstimate-general} 
to \eqref{eq-ch3-equation-u-t} and \eqref{eq-ch3-normal-equation-condition}, 
we conclude, for any $\tau\le \ell-1$ and any $r\in (0,1)$, 
\begin{equation*}
D_{x'}^\tau\partial_t  u,\, tDD_{x'}^\tau \partial_t u,\, t^2D^2D_{x'}^\tau \partial_t u
\in C^\alpha(\bar G_r),\end{equation*}
with $D_{x'}^\tau\partial_t  u=D_{x'}^\tau u_1$, $tDD_{x'}^\tau\partial_t  u=0$, 
and $t^2D^2D_{x'}^\tau\partial_t u=0$ on $\Sigma_1$,
and 
\begin{align*}
&|\partial_t D_{x'}^\tau u|_{C^\alpha(\bar G_{1/2})}+|tD\partial_tD_{x'}^\tau u|_{C^\alpha(\bar G_{1/2})}
+|t^2D^2\partial_tD_{x'}^\tau u|_{C^\alpha(\bar G_{1/2})}\\
&\qquad\le C\big\{|u|_{L^\infty(G_1)}+|f|_{C^{\ell,\alpha}(\bar G_1)}\big\}.
\end{align*}
This is 
the desired result for $m=1$.

The proof for general $m$ is based on induction and hence omitted. 
\end{proof}

We point out that $\partial^i_tu(\cdot, 0)$ for $i=1, \cdots, m$ are determined by the equation. 

\smallskip

Theorem \ref{thrm-ch3-Linear-NormalEstimate-general} implies Theorem \ref{thrm-ch3-regularity-higher} 
easily. 

\section{Regularizations}\label{sec-existence-linear-existence}

In this section, we prove the existence of solutions to the Dirichlet problem by 
regularizing the uniformly degenerate elliptic equations. 

\begin{proof}[Proof of Theorem \ref{thrm-ch2-Existence-Dirichlet}] 
First, the maximum principle implies 
the uniqueness of solutions of  \eqref{eq-ch2-basic-equation} and 
\eqref{eq-ch2-Dirichlet} in the $C(\bar\Omega)\cap C^2(\Omega)$-category. 
We note that the equation in  \eqref{eq-ch2-basic-equation} 
is degenerate only along the boundary $\partial\Omega$. 
Hence, the maximum principle is still valid.  

Next, we prove the existence of a 
$C^\alpha(\bar\Omega)\cap 
C^{2,\alpha}(\Omega)$-solution $u$ of   
\eqref{eq-ch2-basic-equation} and \eqref{eq-ch2-Dirichlet} for the given $f$. 
Take an arbitrary $\delta>0$ and set 
$$L_\delta=L+\delta\Delta.$$ 
Then, the operator $L_\delta$ is uniformly 
elliptic in $\bar\Omega$, with $C^\alpha(\bar\Omega)$-coefficients. 
By the Schauder theory, there exists a unique solution 
$u_\delta\in C(\bar\Omega)\cap C^{2,\alpha}(\Omega)$ of 
the Dirichlet problem 
\begin{align}\label{eq-ch2-Dirichlet-epsilon}\begin{split}
L_\delta u_\delta&=f\quad\text{in }\Omega,\\
u_\delta&=\frac{f}{c}\quad\text{on }\partial\Omega.
\end{split}\end{align}
We point out that $\partial\Omega$ and $(f/c)|_{\partial\Omega}$ are not assumed to be $C^{2,\alpha}$. 
We now derive estimates of $u_\delta$, independent of $\delta$. 
The proof consists of three steps. For convenience, we set 
$$F=|f|_{C^\alpha(\bar\Omega)}.$$

{\it Step 1.} We derive an $L^\infty$-estimate of $u_\delta$. 
In fact, a standard argument based on the maximum principle yields 
\begin{align*} |u_\delta|_{L^\infty(\Omega)}\le C|f|_{L^\infty(\Omega)}
\le CF.\end{align*}

{\it Step 2.} We derive a decay estimate near the boundary. 
We can construct homogeneous supersolutions for $L_\delta$ similarly as in 
Lemma \ref{lemma-ch3-Linear-ComparisonFunction-t-factor} and prove, 
for any $x\in\Omega$ and any $x_0\in \partial\Omega$, 
\begin{equation*}\label{eq-degree-alpha-0a}|u_\delta(x)-u_\delta(x_0)|\le CF|x-x_0|^\alpha.\end{equation*}

{\it Step 3.} We derive a weighted $C^{2,\alpha}$-estimate of $u_\delta$. 
Similarly as Theorem \ref{thrm-ch3-Linear-Regularity-0-alpha}, we prove 
$u_\delta, \rho\nabla u_\delta, \rho^2\nabla^2 u_\delta\in C^\alpha(\bar\Omega)$, 
with  $\rho\nabla u_\delta=0$ and $\rho^2\nabla^2 u_\delta=0$ on $\partial\Omega$, 
and 
\begin{equation*}\label{eq-estimate-u-epsilon-boundary}
|u_\delta|_{C^\alpha(\bar\Omega)}+|\rho\nabla u_\delta|_{C^\alpha(\bar\Omega)}
+|\rho^2\nabla^2 u_\delta|_{C^\alpha(\bar\Omega)}\le CF.\end{equation*}

We point out that the positive constants $C$ in Steps 1-3 
are independent of $\delta$. As a consequence, 
there exists a sequence $\delta=\delta_i\to 0$ such that 
$$u_\delta\to u,\quad \rho\nabla u_\delta\to u^{(1)}, \quad
\rho^2\nabla^2u_\delta\to u^{(2)}\quad\text{uniformly in }\bar\Omega,$$
for some $u\in C^\alpha(\bar\Omega)$, $u^{(1)}\in C^\alpha(\bar\Omega; \mathbb R^n)$, 
and $u^{(2)}\in C^\alpha(\bar\Omega; \mathbb R^{n\times n})$, 
with $u=f/c$, $u^{(1)}=0$, and $u^{(2)}=0$ on $\partial\Omega$. 
A simple argument shows $u\in C^{2,\alpha}(\Omega)$, $\rho\nabla u=u^{(1)}\in C^\alpha(\bar\Omega)$, 
and $\rho^2\nabla^2 u=u^{(2)}\in C^\alpha(\bar\Omega)$, with $\rho\nabla u=0$ and $\rho^2\nabla^2u=0$ 
on $\partial\Omega$. Therefore, 
$$\nabla u_\delta\to \nabla u, \quad
\nabla^2u_\delta\to \nabla^2u\quad\text{uniformly in any compact subsets of }\Omega.$$
By letting $\delta\to 0$ in \eqref{eq-ch2-Dirichlet-epsilon}, we have that $u$ is a solution of 
\eqref{eq-ch2-basic-equation} and \eqref{eq-ch2-Dirichlet}. 
\end{proof}

\end{document}